\documentclass{amsart}

\usepackage{amsmath,amssymb,amsthm,amsfonts,mathrsfs,textcomp,mathtools,accents}
\usepackage{graphicx}
\usepackage{enumerate}
\usepackage[shortlabels]{enumitem}
\setlist[1]{topsep=0.2cm,itemsep=0.2cm}

\usepackage[headings]{fullpage}
\usepackage{hyperref}
\hypersetup{
  colorlinks=false
}

\DeclareFontFamily{U}{mathx}{\hyphenchar\font45}
\DeclareFontShape{U}{mathx}{m}{n}{<-> mathx10}{}
\DeclareSymbolFont{mathx}{U}{mathx}{m}{n}
\DeclareMathAccent{\widebar}{0}{mathx}{"73}

\usepackage[figuresright]{rotating}
\usepackage{tikz}
\usetikzlibrary{arrows,calc,positioning}

\newtheorem{theorem}{Theorem}[section]
\newtheorem{corollary}[theorem]{Corollary}
\newtheorem{lemma}[theorem]{Lemma}
\newtheorem{proposition}[theorem]{Proposition}

\theoremstyle{definition}
\newtheorem{definition}[theorem]{Definition}

\theoremstyle{remark}
\newtheorem{remark}[theorem]{Remark}

\numberwithin{equation}{section}

\newcommand{\norm}[1]{\left\Vert#1\right\Vert}
\newcommand{\abs}[1]{\left\vert#1\right\vert}
\newcommand{\set}[1]{\left\{#1\right\}}

\newcommand{\R}{\mathbb R}
\newcommand{\C}{\mathbb C}
\newcommand{\N}{\mathbb N}
\newcommand{\Z}{\mathbb Z}
\newcommand{\Q}{\mathbb Q}
\newcommand{\CC}{{\mathcal C}}
\newcommand{\BB}{{\mathcal B}}
\renewcommand{\AA}{{\mathcal A}}
\newcommand{\PP}{{\mathcal P}}
\newcommand{\II}{{\mathfrak J}}

\newcommand{\VV}{{\mathcal V}}
\newcommand{\eps}{\varepsilon}
\newcommand{\trianglecl}{\widebar{\triangle}}
\newcommand{\couple}[4]{\left(\frac{#1}{#2},\frac{#3}{#4}\right)}
\newcommand{\triple}[3]{\left(\frac{#1}{#3},\frac{#2}{#3}\right)}
\newcommand{\singx}[3]{\left({#1},\frac{#2}{#3}\right)}
\newcommand{\singy}[3]{\left(\frac{#1}{#2},{#3}\right)}

\DeclareMathOperator*{\esssup}{ess\,sup}
\DeclareMathOperator*{\essinf}{ess\,inf}
\DeclareMathOperator*{\diam}{diam}
\DeclareMathOperator*{\Prob}{Prob}
\DeclareMathOperator*{\ind}{\textbf{1}}

\usepackage{xcolor}
\definecolor{dgreen}{RGB}{0,160,0}    
\definecolor{orange}{RGB}{251,111,66} 
\definecolor{lblue}{RGB}{8,180,238}   
\definecolor{DBlu}{rgb}{.1,.1,.7}     

\def\centerarcfill[#1](#2)(#3:#4:#5){\draw[#1] (#2) --
  ($(#2)+({#5*cos(#3)},{#5*sin(#3)})$) arc (#3:#4:#5);}

\begin{document}

\title[A slow triangle map and a tree of rational pairs]{A slow triangle map
  with a segment of indifferent fixed points and a complete tree of
  rational pairs}%

\author{Claudio Bonanno}%
\address{Dipartimento di Matematica, Universit\`a di Pisa, Largo Bruno
  Pontecorvo 5, 56127 Pisa, Italy} \email{claudio.bonanno@unipi.it}

\author{Alessio Del Vigna}%
\address{Dipartimento di Matematica, Universit\`a di Pisa, Largo Bruno
  Pontecorvo 5, 56127 Pisa, Italy} \email{delvigna@mail.dm.unipi.it}

\author{Sara Munday}%
\address{Dipartimento di Architettura, Universit\`a degli studi Roma
  Tre, via Madonna dei Monti 40, 00184, Roma}
\email{saraann.munday@uniroma3.it}

\thanks{We thank the referee for useful comments and for pointing out a mistake in the previous version.
This research is part of the authors' activity within the
  DinAmicI community, see \url{www.dinamici.org}. The authors are
  partially supported by the research project
  PRA$\textunderscore$2017$\textunderscore$22 ``Dynamical systems in
  analysis, geometry, mathematical logic and celestial mechanics'' of
  the University of Pisa. C. Bonanno and S. Munday are partially
  supported by the Istituto Nazionale di Alta Matematica and its
  division Gruppo Nazionale di Fisica Matematica. C. Bonanno is partially supported also by the research
  project PRIN 2017S35EHN$\textunderscore$004 ``Regular and stochastic
  behaviour in dynamical systems'' of the Italian Ministry of
  Education and Research.}

\begin{abstract}
    We study the two-dimensional continued fraction algorithm
    introduced in \cite{garr} and the associated \emph{triangle map}
    $T$, defined on a triangle $\triangle\subseteq \R^2$. We introduce
    a slow version of the triangle map, the map $S$, which is ergodic
    with respect to the Lebesgue measure and preserves an infinite
    Lebesgue-absolutely continuous invariant measure. We discuss the properties that the
    two maps $T$ and $S$ share with the classical Gauss and
    Farey maps on the interval, including an analogue of the
    weak law of large numbers and of Khinchin's weak law for the
    digits of the triangle sequence, the expansion associated to $T$. Finally, we confirm
    the role of the map $S$ as a two-dimensional version of the Farey
    map by introducing a complete tree of rational pairs,
    constructed using the inverse branches of $S$, in the same way as the
    Farey tree is generated by the Farey map, and then, equivalently,
    generated by a generalised mediant operation.
\end{abstract}

\maketitle

\section{Introduction}
\label{sec:intro}

The theory of (regular) continued fractions has
received much attention from researchers in ergodic theory in the last decades, most recently
thanks to the development of infinite ergodic theory
(\cite{aaronson:iet,IK,isola,munday:iet}). For instance, the general
results of ergodic theory have been applied to the Gauss and the Farey
maps to obtain new proofs of the Gauss-Kuzmin Theorem,  Khinchin's weak law and
other metric results first obtained by Khinchin and L\'evy.

One of the most notable results in the theory of continued
fractions is Lagrange's Theorem, which states that a real number has an
eventually periodic continued fraction expansion if and only if it is a quadratic
irrational. In a letter to Jacobi, Hermite asked
whether it was possible to obtain a similar classification for the
algebraic irrationals of higher degree. It was for this reason that
Jacobi developed what is now called the Jacobi-Perron algorithm, and the
theory of multidimensional continued fractions began. Unfortunately,
despite numerous attempts and the introduction of  many different algorithms,
Hermite's question remains unanswered. We refer
the reader to \cite{brentjes} for a geometric description of the
theory of multidimensional continued fractions and to \cite{schw-book}
for some applications of ergodic theory in this area.

In this paper we consider the two-dimensional version of the continued
fraction algorithm introduced in \cite{garr}. The algorithm, which we
describe in Section~\ref{sec:ts}, is based on the iteration of a map
$T$ defined on a triangle $\triangle\subseteq \R^2$, and for this reason, $T$ is
referred to as the \emph{triangle map} and the expansions obtained through
this method are called \emph{triangle sequences}. The ergodic
properties of $T$ are studied in \cite{nog,garr2}; in particular, it is shown
that the map $T$ is ergodic with respect to the Lebesgue
measure on $\triangle$ and preserves a Lebesgue-absolutely continuous
probability measure. The triangle map behaves similarly to the Gauss
map in many ways, for instance, the triangle map acts on triangle sequences by
left-shifting the digits, exactly as the Gauss map does for the
regular continued fraction expansions.

The similarity between the two maps is strengthened by the results of
this paper. We introduce a map $S$ on the triangle $\triangle$, which
plays for $T$ the same role that the Farey map plays for the Gauss map. For
this reason we call $S$ a \emph{slow triangle map}. From the point of
view of ergodic theory, it is interesting to notice that the map $S$
is a piecewise linear fractional map on a finite partition with a segment of
indifferent fixed points, that is, points for which the determinant
of the Jacobian is $1$, and that it is non-uniformly expanding
elsewhere. We show that, similarly to the Farey map, $S$ preserves an
infinite Lebesgue-absolutely continuous measure and it is ergodic with respect
to the Lebesgue measure on $\triangle$. It follows that the
statistical behaviour of summable observables along orbits of $S$ is
non-standard. This phenomenon, for the Farey map, makes it impossible to
improve Khinchin's weak law for the coefficients of the regular
continued fraction expansion to a strong law. However, we are able to exploit certain results from infinite
ergodic theory to show that the system generated by $S$ is pointwise
dual ergodic and, under a further assumption, prove a weak law of large numbers for $S$, from
which we obtain an analogue of Khinchin's weak law for the
digits of the triangle sequences.

The connection between the Gauss and the Farey maps and the regular
continued fractions can be studied also through the Farey
tree, a binary tree which contains all the rational numbers in $(0,1)$
(see \emph{e.g.} \cite{BI}). The Farey tree is strongly related to the
Farey map, but it can also be defined through the mediant operation
on fractions. We recall the definition of the Farey tree and its basic
properties in Section~\ref{sec:tree}. Analogously,
in this paper we define a tree of rational pairs, first by using a
suitable modification of the map $S$ limited to the set of
indifferent fixed points, and then by using a generalised mediant operation
defined on pairs of rational numbers. We prove that the two trees are
in fact identical level by level, and that the tree is complete, that
is, it contains every pair of rational numbers in $\trianglecl$. This
last result improves on the results of \cite{amburg:stern_seq}, where the authors
study different trees generated by the triangle map and its
generalisations, but show that none of them are complete.

The paper is organised as follows. In Section~\ref{sec:setting} we
recall the definition of the triangle map $T$ and the associated
two-dimensional continued fraction algorithm. We also introduce the
map $S$ and study its basic ergodic properties. Lastly, we define a
dynamical system on an infinite strip, which is isomorphic to the
action of $S$ on $\trianglecl$. This isomorphism gives a useful intuitive representation
of the action of $S$ and simplifies some
computations. Section~\ref{sec:wlln} contains the main results on the
ergodic properties of $S$. We prove that the map $S$ is
pointwise dual ergodic with respect to a sequence $a_n(S) \asymp \frac{n}{\log^2 n}$, and use various results from Infinite Ergodic Theory (see \cite{aaronson:iet,munday:iet}) to show that if the sequence $a_n(S)$ is regularly varying (see \eqref{def:reg-var}) then we have the weak law of large
numbers for summable observables (Theorem~\ref{thm-wlln}), and a
Khinchin-type weak law for the triangle sequences
(Corollary~\ref{cor-kuzdig}). The technical results are proved
in Appendix~\ref{sec:matrices} and~\ref{sec:varying-seq}. In
Section~\ref{sec:ba} we apply a result from \cite{LM} to our map
$S$. Recalling that the behaviour of Birkhoff sums of summable
observables drastically changes in infinite ergodic theory, following
\cite{limix}, in \cite{LM} the authors give a Birkhoff Ergodic Theorem
for non-summable observables for infinite-measure-preserving dynamical
systems. We use the version of $S$ defined on the strip and prove a
pointwise convergence theorem for non-summable observables. Finally,
in Section~\ref{sec:tree} we introduce the tree of rational pairs
produced by the counterimages of $\tilde S$, a slightly modified
version of $S$. In Theorem~\ref{thm:tree-compl} we prove that the tree
is complete and that each pair of rationals appears exactly once. Then
we introduce an algorithm on the triangle $\triangle$, based on the
notion of mediant of two fractions, and show in
Theorem~\ref{thm:tree-mediant} that the tree can be generated also by
this algorithm. This concludes the similarity between the slow
triangle map $S$ and the Farey map. For these reasons $S$ may be
considered a \emph{two-dimensional Farey map}. Many interesting
questions remain open about the tree and its connections with the map
$S$ and with the approximation of irrational pairs by rational
pairs. These problems will be subject of future research.

\section{The setting}
\label{sec:setting}

As anticipated in the introduction, the main goal of this paper is to
investigate a two-dimensional map related to the triangle map $T$, as
introduced in \cite{garr}. Let us first recall the definition of the
map $T:\triangle\rightarrow\trianglecl$, where $\triangle$ is the
triangle
\[
    \triangle \coloneqq \set{(x,y)\in \R^2 \,:\, 1 \ge x \ge y > 0}.
\]
Consider the countable partition $\{\triangle_k\}_{k\ge 0}$ of
$\triangle$ into disjoint triangles
\[
    \triangle_k \coloneqq \set{(x,y)\in \triangle \,:\, 1-x-ky \ge 0 >
      1-x-(k+1)y},
\]
shown in Figure~\ref{fig:partition}, and the segment
$\Lambda \coloneqq \set{0\leq x\leq 1,\,y=0}$. Note that
$\trianglecl = \bigcup_{k\ge 0} \triangle_k \cup \Lambda$. The
triangle map $T:\triangle \to \trianglecl$ is then defined to be
\[
    \label{eq-triangle}
    T(x,y) \coloneqq \left( \frac yx, \frac{1-x-ky}{x} \right) \quad
    \text{for }(x,y)\in \triangle_k.
\]
We now define a map $S:\trianglecl \rightarrow \trianglecl$ that can
be thought of as a ``slow version'' of the map $T$. Let us start with
the partition $\{\Gamma_0, \Gamma_1\}$ of $\trianglecl$ (see
Figure~\ref{fig:partition}), where
\[
    \Gamma_0 \coloneqq \triangle_0 = \set{(x,y)\in \R^2 \,:\, 1 \ge
    x \ge y > 1-x},
\]
and
\[
    \Gamma_1 \coloneqq \trianglecl\setminus \Gamma_0 = \bigcup_{k\ge
    1} \triangle_k \cup \Lambda = \set{(x,y)\in \R^2 \,:\,
    1-y \ge x\ge y\ge 0}.
\]
We define $S: \trianglecl \rightarrow \trianglecl$ by setting
\begin{equation}\label{eq-slow}
    S(x,y) \coloneqq
    \begin{cases}
        \left(\frac yx,\,\frac{1-x}{x}\right) & \text{if }(x,y)\in\Gamma_0\\[0.2cm]
        \left(\frac{x}{1-y},\,\frac{y}{1-y}\right) & \text{if
        }\,(x,y)\in\Gamma_1
    \end{cases}.
\end{equation}

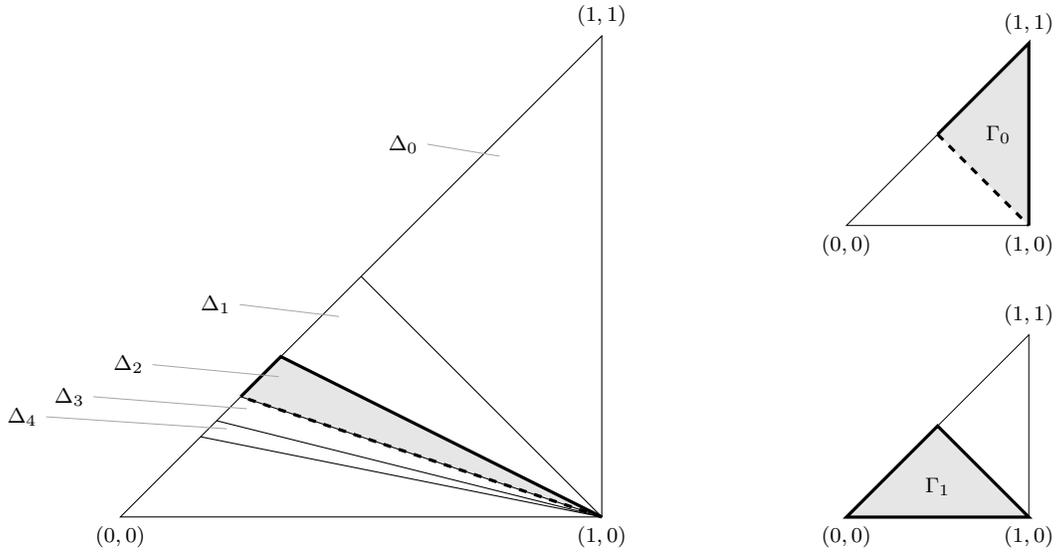
\begin{figure}[h]
    \begin{tikzpicture}[scale=3.2]
    \def\s{2}
    \def\o{0.075}

    \draw[thin] (0,0) node[below] {\footnotesize $(0,0)$} -- (\s,0)
    node[below] {\footnotesize $(1,0)$} -- (\s,\s) node[above]
    {\footnotesize $(1,1)$}-- cycle;

    \draw[very thick,fill=gray!20] (\s,0) -- (1/3*\s,1/3*\s) -- (1/4*\s,1/4*\s);
    \draw[very thick,dashed] (\s,0) -- (1/4*\s,1/4*\s);

    \foreach \x in {2,3,4,5,6}
    {
      \draw[thin] (1/\x*\s,1/\x*\s) -- (\s,0);
    }

    \foreach \x in {0,1,2,3,4}
    {
      \draw[very thin,gray!70] ({1/2*(1/(\x+1)+1/(\x+2))*\s+\o},{1/2*(1/(\x+1)+1/(\x+2))*\s}) --
      ({1/2*(1/(\x+1)+1/(\x+2))*\s-1.5*(\x+2)*\o},{1/2*(1/(\x+1)+1/(\x+2))*\s+1/2\o})
      node[black,left] {\footnotesize $\Delta_\x$};
    }
\end{tikzpicture}
    \hspace{2cm}
    \begin{tikzpicture}[scale=2.425]
    \def\t{-1.6}
    \draw[very thin] (0,0) node[below] {\footnotesize $(0,0)$} -- (1,0)
    node[below] {\footnotesize $(1,0)$} -- (1,1) node[above] {\footnotesize $(1,1)$} -- cycle;
    \draw[very thick,fill=gray!20] (1,0) -- (1,1) -- (1/2,1/2);
    \draw[very thick,dashed] (1/2,1/2) -- (1,0);
    \node at (5/6,1/2) {\footnotesize $\Gamma_0$};

    \draw[very thin] (0,0+\t) node[below] {\footnotesize $(0,0)$} -- (1,\t)
    node[below] {\footnotesize $(1,0)$} -- (1,1+\t) node[above] {\footnotesize $(1,1)$} -- cycle;
    \draw[very thick,fill=gray!20] (1,\t) --
    (1/2,1/2+\t) -- (0,\t) -- cycle;
    \node at (1/2,1/6+\t) {\footnotesize $\Gamma_1$};
\end{tikzpicture}
    \caption{\emph{Left}. Partition of $\triangle$ into
      $\{\triangle_k\}_{k\geq 0}$. \emph{Right}. Partition of
      $\trianglecl$ into $\Gamma_0$ and
      $\Gamma_1$.}\label{fig:partition}
\end{figure}

The relation between these two maps is that the triangle map $T$ is
the jump transformation of $S$ on the set $\Gamma_0$. In other words,
if we introduce the first passage time function
\[
    \tau(x,y) \coloneqq 1+\min \set{k\ge 0 \,:\, S^k(x,y)\in\Gamma_0},
\]
then it can be readily calculated that $T(x,y) = S^{\tau(x,y)}(x,y)$
for each $(x,y) \in \triangle$. Notice that
$S(\triangle_k) = \triangle_{k-1}$ for $k\ge 1$, and that
$S(\Gamma_0) \cup \{x=y,\, 0\le x\le 1\} =
S(\Gamma_1)=\trianglecl$. Moreover the segment $\Lambda$ consists of
fixed points, that is $S(x,0) = (x,0)$ for $0\leq x\leq 1$. The
determinant of the Jacobian of $S$ turns out to be
\[
    JS(x,y) =
    \begin{cases}
        \frac{1}{x^3}     & \text{if }(x,y)\in\Gamma_0\\
        \frac{1}{(1-y)^3} & \text{if }(x,y) \in \Gamma_1
    \end{cases},
\]
and it follows that $JS(x,0) = 1$ for $0\leq x\leq 1$. Thus the
segment $\Lambda$ consists of indifferent fixed points.

\subsection{Invariant measure and the transfer operator}
\label{sec:imto}
In \cite{nog} it is shown that the map $T$ is ergodic, and from
\cite{garr2} we know that the unique ergodic, Lebesgue-absolutely
continuous $T$-invariant probability measure on $\triangle$ is given
by the density \[ k(x,y) = \frac{12}{\pi^2x(1+y)}.
\]
Applying classical results from ergodic theory
(\cite{aaronson:iet,munday:iet}), the existence of an ergodic,
Lebesgue-absolutely continuous $S$-invariant measure immediately
follows. One way to find the density $h(x,y)$ of this measure is to
look for a fixed point of the transfer operator $\PP$ associated to
$S$. Let
\[
    \phi_0 \coloneqq (S|_{\Gamma_0})^{-1} : \trianglecl\setminus
    \{x=y,\,0\le x\le 1\} \rightarrow \Gamma_0,\quad \phi_0(x,y) =
    \left(\frac{1}{1+y},\ \frac{x}{1+y} \right)
\]
and
\[
    \phi_1\coloneqq (S|_{\Gamma_1})^{-1} : \trianglecl \rightarrow
    \Gamma_1,\quad \phi_1(x,y) = \left( \frac{x}{1+y},\
      \frac{y}{1+y}\right)
\]
be the local invers maps of $S$. The transfer operator $\PP$ is then
defined for each measurable function $f$ on $\trianglecl$ by setting
\begin{align*}
    (\PP f)(x,y) & = |J\phi_0(x,y)|f(\phi_0(x,y)) + |J\phi_1(x,y)|f(\phi_1(x,y))=\\
    & = \frac{1}{(1+y)^3}f\left( \frac{1}{1+y},\ \frac{x}{1+y} \right)
    +\frac{1}{(1+y)^3} f \left( \frac{x}{1+y},\ \frac{y}{1+y} \right).
\end{align*}
A straightforward computation shows that $\PP h = h$ for
$h(x,y) = \frac{1}{xy}$.

\begin{proposition}\label{prop-mis-inv}
    The system $(\trianglecl, \mu, S)$ is conservative, and the map
    $S$ admits a unique, up to multiplicative constants, ergodic
    invariant measure $\mu$, absolutely continuous with respect to the
    Lebesgue measure $m$, given by the density
    $h(x,y) = \frac{1}{xy}$. The measure $\mu$ is $\sigma$-finite and
    $\mu(\trianglecl)= +\infty$.
\end{proposition}

\begin{proof}
    For conservativity, in light of Maharam's Recurrence Theorem
    \cite[Theorem~2.2.14]{munday:iet}, it is enough to observe that
    \[
        \triangle = \bigcup_{n=0}^{\infty}\, S^{-n} (\Gamma_0)\pmod{\mu},
    \]
    which is a consequence of the fact that
    $\triangle_k\subseteq S^{-k} (\Gamma_0)$, for each $k\ge0$. We
    have already discussed the existence of the measure $\mu$ above.
    That $\mu$ is unique follows, for example, from
    \cite[Theorem~2.4.35]{munday:iet}, on noting that $S$ is
    conservative, ergodic, and is certainly non-singular with respect
    to $m$.  Finally, that $\mu$ is $\sigma$-finite follows from
    computing the measure of the triangles $\triangle_k$ as in
    \cite{garr2}.
\end{proof}

\subsection{An equivalent system on a strip}
\label{sec:ess}
For later use, we now introduce another system, isomorphic to the map
$S$. Using the change of coordinates defined on $\triangle$ by
\[
    (x,y) \mapsto (u,v) \in \Sigma \coloneqq (0,1] \times [0,+\infty),
    \qquad u(x,y) = \frac yx,\quad v(x,y) = \frac{1-x}{y},
\]
one can show that the system $(\triangle, \mu, S)$ is isomorphic
mod~$\mu$ to the system $(\Sigma, \rho, F)$ given by
\[
    F(u,v) =
    \begin{cases}
        \left(v,\frac 1v(\frac 1u -1)\right) & %
        \text{if }(u,v) \in \Pi_0 \coloneqq \{ (u,v)\in\Sigma \,:\, v<1\} \\[0.2cm]
        \left(u,v-1\right) & %
        \text{if }(u,v)\in\Pi_1 \coloneqq \{(u,v)\in\Sigma \,:\, v\ge
        1\}
    \end{cases},
\]
with $d\rho(u,v) = \frac{1}{1+uv}du\,dv$. The sets
$\Pi_0$ and $\Pi_1$ partition the strip $\Sigma$ and correspond
mod~$\mu$ to $\Gamma_0$ and $\Gamma_1$, respectively. We also
introduce the countable partition $\{\Sigma_k\}_{k\geq 0}$, where
\[
    \Sigma_k \coloneqq \set{(u,v)\in \Sigma\, :\, k\le v < k+1}
\]
is a unit squares, as shown in Figure~\ref{fig:sigma_k}. Note that
$\Sigma=\bigcup_{k\geq 0} \Sigma_k$. This is the analogue of the
partition $\{\triangle_k\}_{k\geq 0}$ of the triangle $\triangle$. The
local inverses of the map $F$ are given by
\[
  (F|_{\Pi_0})^{-1} (u,v) = \left(\frac{1}{uv+1},u\right)
  \quad\text{and}\quad
  (F|_{\Pi_1})^{-1} (u,v) = \left(u,v+1\right),
\]
so that the transfer operator $\PP_F$ associated to $F$ turns out to be
\[
    (\PP_F g) (u,v) = \frac{u}{(uv+1)^2} g\left(\frac{1}{uv+1},u\right)
    + g\left(u,v+1\right).
\]
It can be immediately verified that the density of the measure $\rho$
is a fixed point of $\PP_F$.

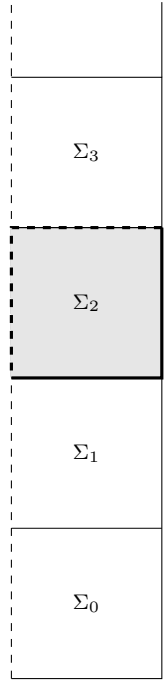
\begin{figure}[h]
    \begin{tikzpicture}[scale=2]
    \def\h{4}
	\def\s{2}

    \draw[thin] (0,0) -- (1,0) -- (1,\h) -- (0,\h);
    \draw[thin,dashed] (0,0) -- (0,\h+1/2);
    \draw[thin] (1,\h) -- (1,\h+1/2);

    \draw[fill=gray!20] (0,\s) -- (1,\s) -- (1,\s+1) -- (0,\s+1);
    \draw[very thick] (0,\s) -- (1,\s) -- (1,\s+1);
    \draw[very thick,dashed] (1,\s+1) -- (0,\s+1) -- (0,\s);

    \foreach \x in {1,2,3}
    {
      \draw[very thin] (0,\x) -- (1,\x);
    }

    \foreach \x in {0,1,2,3}
    {
      \node[black] at (1/2,\x+1/2) {\footnotesize $\Sigma_\x$};
    }
\end{tikzpicture}
    \caption{Partition of the strip $\Sigma$ into
      $\{\Sigma_k\}_{k\geq 0}$.}\label{fig:sigma_k}
\end{figure}

\subsection{Triangle sequences}
\label{sec:ts}

Let us now recall the definition of the \emph{triangle sequence}
associated to a point $(x, y)$ in $\triangle$ and certain results
concerning their digits from \cite{garr}. We start by setting
$d_{-2} \coloneqq 1$, $d_{-1}\coloneqq x$ and $d_0\coloneqq y$ and,
supposing that $d_{k-3}>d_{k-2}>d_{k-1}>0$, we recursively define
$\alpha_k\coloneqq\alpha_k(x, y)$ to be the non-negative
integer such that
\[
    d_{k-3} - d_{k-2} -\alpha_k d_{k-1}\geq0
\]
and
\[
    d_{k-3} - d_{k-2} -(\alpha_k+1) d_{k-1}<0.
\]
Then set $d_k\coloneqq d_{k-3} - d_{k-2} -\alpha_k d_{k-1}\in
\R^+$. If at any stage we find that $d_k=0$, the process stops. We
shall write $(x, y) = (\alpha_1,\,\alpha_2,\,\ldots)$ to denote the
triangle sequence of $(x,y)$. Another way of defining the triangle
sequence is to note that $\alpha_k(x,y)=m$ if and only if
$T^{k-1}(x,y)\in \triangle_{m}$, and the process stops if
$T^n(x,y) \in \Lambda$ for some $n\geq 1$. From this way of looking at
the triangle sequence digits, it immediately follows that if
$(x, y)= (\alpha_1,\,\alpha_2,\,\ldots)$, then
$T(x, y)= (\alpha_2,\,\alpha_3,\,\ldots)$. In other words, the
triangle map acts on triangle sequences as the shift map, exactly as
the Gauss map does for the continued fraction expansions. We also have
the following relation between the digits $\alpha_k$ and the first
passage time:
\begin{equation}\label{eq:digit_tau}
    \tau (T^{k-1}(x,y)) = 1 + \alpha_k(x, y).
\end{equation}

In \cite{garr}, the following results for the triangle sequence are
given.
\begin{itemize}
  \item If $(x, y)$ is a pair of rational numbers in
    $\Q^2\cap \trianglecl$, then the triangle sequence associated to
    $(x, y)$ is finite. However, the converse is not true:
    non-rational points can also have finite triangle sequences.
  \item Every infinite sequence of non-negative integers
    $(\alpha_1,\, \alpha_2,\, \ldots)$ has a pair $(x, y)\in\triangle$
    that has this sequence as its triangle sequence.
  \item If an integer $k$ appears infinitely often in a given sequence
    of integers, there is a unique pair $(x, y)\in\triangle$ that has
    this sequence as its triangle sequence.
\end{itemize}
Note that there are entire line segments with every point having
identical infinite triangle sequences. This is essentially due to the
fact that the refinements of the partition $\{\triangle_k\}_{k\geq0}$
with respect to the map $T$ do not have diameters shrinking to
$0$. Thus, whilst the triangle sequence can usefully be thought of as
a two-dimensional generalisation of the continued fraction expansion,
in certain respects it behaves rather differently. However, this
behaviour is not in contrast with the ergodicity of the map, since as
shown in \cite{nog} for Lebesgue almost every point the refinements of
the partition $\{\triangle_k\}_{k\geq0}$, along the triangle sequence
of the point, shrink to the point. In the language of multidimensional
continued fraction expansions (see \cite{brentjes}), this means that
the triangle sequence is \emph{weakly convergent} at Lebesgue almost
every point.

\section{A weak law of large numbers}
\label{sec:wlln}

For dynamical systems with an infinite invariant measure, in general
it is only possible to establish weaker statistical properties than
those for systems with an invariant probability measure.  For example,
if $(X, \mu, R)$ is a conservative and ergodic measure-preserving
system such that $\mu(X)=\infty$, then Birkhoff's Ergodic Theorem
becomes the weak statement that
\[
    \lim_{n\to\infty}\frac 1n \sum_{k=0}^{n-1} (f\circ R^k)(x) = 0,
\]
for $\mu$-almost every $x\in X$ and for all $f\in
L^1(X,\mu)$. Moreover, the exact asymptotic pointwise behaviour cannot
be recovered for all $f\in L^1(X,\mu)$ by changing the normalising
sequence, due to Aaronson's Ergodic Theorem \cite[Theorem
2.4.2]{aaronson:iet}, which basically states that for any sequence of
positive real numbers, the growth rate of the Birkhoff sums will be
either over- or under-estimated infinitely often.  Nevertheless, it is
possible to obtain distributional limit laws for the ergodic sums of
some classes of dynamical systems with an infinite invariant measure
(see \cite[Chapter 3]{aaronson:iet}).

A first step is to show that that the system $(\trianglecl, \mu, S)$ is pointwise dual ergodic,
which means that there exists a sequence
$(a_n(S))_{n\geq 0}$ such that
\[
  \lim_{n\rightarrow \infty}\frac{1}{a_n(S)} \sum_{k=0}^{n-1} (\PP^k
  f)(x,y) = \int_{\trianglecl} fd\mu
\]
for $\mu$-almost every $(x,y)\in \trianglecl$ and for all
$f\in L^1(\trianglecl,\mu)$, where $\PP$ is the transfer operator of
the system. We prove it for $(\trianglecl, \mu, S)$.

\begin{theorem}
  \label{thm-pde}
  The system $(\trianglecl, \mu, S)$ is pointwise dual ergodic, and the
  sequence $\left(a_n(S)\right)_{n\geq 0}$ satisfies\footnote{We say that $a_n\asymp b_n$ if and only if
    $a_n =O(b_n)$ and $b_n=O(a_n)$.} $a_n(S) \asymp \frac{n}{\log^2 n}$.
\end{theorem}

Distributional limit laws follow from pointwise dual ergodicity under the assumption that the sequence $(a_n(S))_{n\geq 0}$ is regularly varying. We recall that a sequence $(a_n)_{n\geq 0}$ is said to be
\emph{regularly varying of index $\alpha\in \R$} if for all $c>0$ we
have that
\begin{equation}\label{def:reg-var}
    \lim_{n\to \infty} \frac{a_{\lfloor cn\rfloor}}{a_n} = c^\alpha.
\end{equation}
If $\alpha=0$ the sequence is called \emph{slowly varying}.

\begin{theorem}[Weak law of large numbers]
  \label{thm-wlln}
  Let $\Prob$ be a probability measure on $\trianglecl$, absolutely
  continuous with respect to the Lebesgue measure. If the sequence $(a_n(S))_{n\geq 0}$ in Theorem \ref{thm-pde} is regularly varying of index $\alpha=1$, then for all
  $f\in L^1(\trianglecl,\mu)$ and for all $\eps>0$
  \[
    \lim_{n\rightarrow \infty} \Prob \left( \left| \frac{1} {a_n(S)}
        \sum_{k=0}^{n-1}(f\circ S^k)(x,y) - \int_{\trianglecl} f
        d\mu \right|> \eps\right) =0.
  \]
\end{theorem}

\begin{corollary}[Khinchin weak law]
  \label{cor-kuzdig}
  Let $\Prob$ be a probability measure on $\trianglecl$, absolutely
  continuous with respect to the Lebesgue measure. If the sequence $(a_n(S))_{n\geq 0}$ in Theorem \ref{thm-pde} is regularly varying of index $\alpha=1$, then there exists a
  sequence $\left(b_n\right)_{n\geq 0}$ such that for all $\eps>0$
  \[
      \lim_{n\rightarrow \infty} \Prob \left( \left|\frac{1}{b_n}
          \sum_{k=0}^{n-1}\alpha_k(x,y) - 1 \right|> \eps\right) =0,
  \]
  and $b_n \asymp n\log^2 n$. In particular, for $m$-almost every
  $(x,y)\in \trianglecl$
  \[
      \lim_{n\rightarrow \infty} \frac 1n \sum_{k=0}^{n-1}
      \alpha_k(x,y) = +\infty.
  \]
\end{corollary}
\begin{proof}
    Applying Theorem \ref{thm-wlln} to the function
    $f=\ind_{\triangle_0}$, the proof follows from a standard duality
    argument between Birkhoff sums and the return time function (see
    for example \cite[pag. 22]{zwei}), which is related to the
    triangle sequence by \eqref{eq:digit_tau}.
\end{proof}

Given Theorem \ref{thm-pde}, the proof of Theorem \ref{thm-wlln} is
then completed by appealing to the Darling-Kac theorem, which implies
that the distributional limit of the Birkhoff sums
$\frac{1}{a_n(S)} \sum_{k=0}^{n-1} f\circ S^k$ is
$\left(\int_{\trianglecl} f d\mu\right)\mathcal{M}_1$ for all
$f\in L^1(\trianglecl,\mu)$, where $\mathcal{M}_1$ is the random
variable with normalised Mittag-Leffler distribution of order
$\alpha=1$, (see \cite[Corollary 3.7.3]{aaronson:iet}). In particular,
since $\mathcal{M}_1$ is constant, the Birkhoff sums converge in
probability and Theorem \ref{thm-wlln} is proved.

\subsection{Proof of Theorem \ref{thm-pde}}
\label{subsec:proof_pde}

We first recall the results we use to prove that the system
$(\trianglecl,\mu,S)$ is pointwise dual ergodic.

\begin{definition}
  \label{def:cfm}
  Let $V$ be a measure-preserving transformation of the probability
  space $(\Omega,\AA,\nu)$ and let $\CC \subseteq \BB$ be a countable
  measurable partition which is generating for $V$. Let us denote by
  $\CC^k$, $k\geq 1$, the iterated partitions, that is
  $\CC^k \coloneqq \bigvee_{j=0}^{k-1}\, V^{-j}\CC$. The system
  $(\Omega,\AA,\nu,V,\CC)$ is said to be \emph{$\psi$-mixing} if the
  sequence
  \[
      \psi_n\coloneqq \sup_{C\in \CC^k\atop B\in \AA,\,\nu(B)>0}
      \frac{\left| \nu\left(C\cap V^{-(k+n)} B\right) -
          \nu(C)\nu(B)\right|}{\nu(C)\nu(B)}
  \]
  satisfies $\psi_n \rightarrow 0$ as $n\rightarrow \infty$.
\end{definition}

\begin{remark}
    The property defined above as $\psi$-mixing is often referred to
    as {\em continued fraction mixing} since in particular it is
    satisfied by the Gauss map, see \cite[Theorem 5.2.7]{munday:iet}.
\end{remark}

\begin{proposition}[\cite{aaronson:iet}, Lemma~3.7.4 and
    Proposition~3.7.5]
  \label{prop:cfm-pde}
  Let $R$ be a conservative, ergodic measure-preserving transformation
  of the space $(X,\BB,\mu)$, and let $A\in \BB$ with
  $0<\mu(A)<+\infty$. Define
  \begin{enumerate}[label={\upshape(\roman*)},wide = 0pt,leftmargin=*]
    \item $\varphi_A(x)\coloneqq \inf\{j\ge 1 \,:\, R^j(x)\in A\}$,
      the first return time function to $A$, which is finite for
      $\mu$-almost every $x\in A$;
    \item the induced map\footnote{The induced map $R_A$ is an ergodic
        measure-preserving transformation of the probability space
        $(A,\BB\cap A,\mu|_A)$. See, for instance,
        \cite[Proposition~1.5.2 and~1.5.3]{aaronson:iet}.}
      $R_A:A\rightarrow A$ as $R_A(x)\coloneqq R^{\varphi_A(x)}(x)$
      for $\mu$-almost every $x\in A$.
  \end{enumerate}
  Let $\CC\subseteq \BB\cap A$ be a countable measurable partition
  which generates $\BB$ under $R_A$, such that $\varphi_A$ is
  $\CC$-measurable. If the induced system
  $(A,\BB\cap A, R_A, \mu|_A, \CC)$ is $\psi$-mixing, then the
  original system $(X,\mu,R)$ is pointwise dual ergodic.
\end{proposition}

\noindent To prove the pointwise dual ergodicity of
$(\trianglecl,\mu,S)$ it is thus enough to find a set
$A\subseteq \trianglecl$ of finite positive measure that satisfies the
assumptions of Proposition~\ref{prop:cfm-pde}. The key point of the
previous result is to prove that the induced system is
$\psi$-mixing. To this end, we exploit the properties of the fibred
systems introduced by Schweiger in \cite{schw} and proved by Nakada to
be $\psi$-mixing under some additional conditions
\cite[Theorem~2]{nakada}.

\begin{definition}
    Let $A$ be a compact and connected subset of $\R^d$, with the
    Borel $\sigma$-algebra $\BB$ and let $m$ denote the
    $d$-dimensional normalised Lebesgue measure on $A$. Let $V$ be a
    measurable map of $A$ onto itself. The pair $(A,V)$ is called a
    \emph{fibred system} if it satisfies the following properties.
    \begin{itemize}
      \item[(h1)] There exists a finite or countable measurable
        partition $\CC=\{C_i\}_{i\in \II}$ of $A$ such that the
        restriction of $V$ to $C_i$ is injective for all $i\in \II$.
      \item[(h2)] The map $V$ is differentiable\footnote{In
          \cite{schw} it is only assumed that $V$ is measurable. We
          assume differentiability to simplify the approach to the
          system $(\trianglecl,\mu,S)$.} and non-singular.
    \end{itemize}
    For $i\in \II$, we denote by $\psi_i$ the inverse of the
    restriction $V|_{C_i}$. The cylinder sets of the iterated
    partition $\CC^n = \bigvee_{j=0}^{n-1}\, V^{-j} \CC$ are defined
    inductively to be
    \[
        C_{i_1,\,\dots,\,i_n} = C_{i_1} \cap V^{-1}
        C_{i_2,\,\dots,\,i_n},
    \]
    and we denote by $\psi_{i_1,\,\dots,\,i_n}$ the local inverse of
    $V^n$ restricted to $C_{i_1,\,\dots,\,i_n}$. Note that
    $\psi_{i_1,\,\ldots,\,i_n} = \psi_{i_1}\circ \dots \circ
    \psi_{i_n}$. In order to state the result about the $\psi$-mixing
    property of fibred systems, we introduce the following further
    conditions.
    \begin{itemize}
      \item[(h3)] There exists a sequence $(\sigma(n))_{n\geq 0}$ with
        $\sigma(n)\rightarrow 0$ as $n\rightarrow \infty$ and such
        that
        \[
            \sup_{(i_1,\,\dots,\,i_n)\in \II^n}\,\diam
            C_{i_1,\,\dots,\,i_n}\le \sigma(n).
        \]
      \item[(h4)] There exist a finite number of measurable subsets
        $U_1,\,\ldots,\,U_N$ of $A$ such that for any cylinder
        $C_{i_1,\,\dots,\,i_n}$ of positive measure, there exists
        $U_j$ with $1\le j\le N$ such that $V^n(C_{i_1,\,\dots,\,i_n})
        = U_j$ up to measure-zero sets.
      \item[(h5)] There exists a constant $\lambda\ge 1$ such that
        \[
            \esssup_{V^n(C_{i_1,\,\dots,\,i_n})} |J\psi_{i_1,\,\dots,\,i_n}|
            \le \lambda\,\essinf_{V^n(C_{i_1,\,\dots,\,i_n})}
            |J\psi_{i_1,\,\dots,\,i_n}|
        \]
        where $J\psi_{i_1,\dots,i_n}$ denotes the Jacobian determinant
        of $\psi_{i_1,\,\dots,\,i_n}$.
      \item[(h6)] For any $1\le j\le N$, $U_j$ contains a proper
        cylinder.
      \item[(h7)] There is a constant $r_1 >0$ such that
        \[
            \left|J\psi_{i_1,\,\dots,\,i_n} (p_1) - J\psi_{i_1,\,\dots,\,i_n}
              (p_2) \right| \le r_1 m(C_{i_1,\,\dots,\,i_n}) \| p_1 - p_2\|
        \]
        for any $p_1,p_2 \in U_j$ and all $j$.
      \item[(h8)] There is a constant $r_2 >0$ such that
        \[
            \left\| \psi_{i_1,\,\dots,\,i_n} (p_1) - \psi_{i_1,\,\dots,\,i_n}
              (p_2) \right\| \le r_2 \sigma(n)\| p_1 - p_2\|
        \]
        for any $p_1,p_2 \in U_j$ and all $j$.
      \item[(h9)] Let $\mathcal F$ be a finite partition generated by
        $U_1,\,\dots,\,U_N$ and denote by ${\mathcal F}^c _m$ the
        cylinders in $\CC^m$ that are not contained in any element of
        $\mathcal F$. Then, as $m\rightarrow \infty$
        \[
            \gamma(m) \coloneqq \sum_{C(i_1,\,\dots,\,i_m) \in {\mathcal
                F}^c _m}\, m(C(i_1,\,\dots,\,i_m)) \rightarrow 0.
        \]
    \end{itemize}
\end{definition}

\begin{proposition}[\cite{nakada}, Theorem 2]
    \label{prop:fs-cfm}
    A system $(A,\BB,V,\CC)$ satisfying \textnormal{(h1)-(h9)} admits
    an invariant probability measure $\nu$ and is $\psi$-mixing.
\end{proposition}

The strategy to prove the pointwise dual ergodicity of our system
$(\trianglecl,\mu,S)$ is therefore to find a set
$A\subseteq \trianglecl$ of finite positive measure in such a way that
the induced system satisfies (h1)-(h9). We set
\begin{equation} \label{the-set-A}
    A \coloneqq \set{(x,y)\in \Gamma_0\, :\, S(x,y) \in \Gamma_0}.
\end{equation}
By definition of $S$, the set $A$ is the triangle with vertices
$Q_1=\left(\frac 12, \frac 12\right)$,
$Q_2= \left(\frac 23, \frac 13\right)$ and $Q_3=(1,1)$, with the sides
$Q_1Q_2$ and $Q_2Q_3$ not included. Furthermore, notice that every
point in the interior of $A$ has triangle sequence of the form
$(0,\,0,\,\alpha_3,\,\ldots)$. Let $V$ be the induced map of $S$ on
$A$, that is
\[
    V(x,y) \coloneqq S^{\varphi_A(x,y)} (x,y),
\]
defined for $m$-almost $(x,y)\in A$, and where
$\varphi_A(x,y)= \min\{ j\ge 1\, :\, S^j(x,y) \in A\}$ is finite for
$m$-almost $(x,y)\in A$. Let us first introduce the partition of $A$
given by the level sets of the function $\varphi_A$, that is
$\tilde \CC = \{ \tilde C_k\}_{k\in \N}$ with
\[
    \tilde C_k \coloneqq \{(x,y)\in A\, :\, \varphi_A(x,y)=k\}.
\]
Note that $\tilde C_1$ is the open triangle with vertices $Q_1$, $Q_2$
and $\left( \frac 34,\frac 12\right)$, whereas $\tilde
C_2=\emptyset$. For each set $\tilde C_k$ we introduce the
sub-partition $\{C_{k,\sigma} \,:\, \sigma \in \{0,1\}^k\}$, where
$\sigma$ is the symbolic representation of the orbit
$\{(x,y),\,S(x,y),\,\dots,\,S^{k-1}(x,y)\}$ of a point
$(x,y) \in \tilde C_k$, with respect to the partition
$\{\Gamma_0,\Gamma_1\}$. Thus we consider the countable partition
\[
    \CC \coloneqq \{C_{k,\sigma} \,:\, k\geq 1,\ \sigma\in\{0,1\}^k\}.
\]
The partition $\CC$ is measurable and $V$ is clearly injective on each
cylinder, because $S|_{\Gamma_0}$ and $S|_{\Gamma_1}$ are injective
and points in the same cylinder have the same symbolic orbit in
$\trianglecl$ up to their first return to $A$. Thus, assumptions (h1)
and (h2) are satisfied by the system $(A,\BB,V,\CC)$. Moreover, by the
standard results for induced maps recalled above, the transformation
$V$ preserves the measure $\mu|_A$, which can be normalised to be a
probability measure $\nu$. Some of the remaining assumptions (h3)-(h9)
are trivially verified. By definition, $V^n$ maps each cylinder
$C_{i_1,\,\dots,\,i_n}$ from the iterated partition $\CC^n$ onto $A$,
thus we can choose $N=1$ and $U_1=A$ in order to satisfy assumptions
(h4), (h6) and (h9).

It remains to show that the conditions (h3), (h5), (h7) and (h8) also
hold for our choice of $A$, $V$ and $\CC$. To this end, we prove some
properties of the local inverses of $V^n$. Let $i=(k,\sigma)$, with
$k\ge 1$ and $\sigma=(\sigma_1,\,\ldots,\, \sigma_k)\in \{0,1\}^k$,
and let $C_i$ be a cylinder of our partition. A local inverse
$\psi_i:A \to C_i$ is given by
\[
    \psi_i = \phi_{\sigma_1} \circ \phi_{\sigma_2} \circ \cdots \circ
    \phi_{\sigma_k}.
\]
Note that, for $k=1$, the only possible index is given by
$\sigma=(0)$, and the corresponding local inverse is simply
$\phi_0$. Moreover, by the definition of $A$, the indices
$i=(k,\sigma)$ with $k\ge 3$ satisfy $\sigma_1=\sigma_2=0$, so that
$\psi_i = \phi_0\circ \phi_0 \circ \phi_{\sigma_3} \circ \cdots \circ
\phi_{\sigma_k}$. In this way all local inverses $\psi_i$ are of the
form
\begin{equation}
    \label{form-psi}
    \psi_i(x,y) = \left( \frac{r_1 +s_1 x + t_1 y}{r+sx+ty},\ \frac{r_2
        +s_2 x + t_2 y}{r+sx+ty}\right),
\end{equation}
with non-negative integer coefficients $r_1$, $r_2$, $r$, $s_1$,
$s_2$, $s$, $t_1$, $t_2$, $t$ (where the dependence on $i$ has been
dropped to simplify the notation). A straightforward computation shows
that
\begin{equation}
    \label{form-dpsi}
    D\psi_i (x,y) =
      \begin{pmatrix}
          \dfrac{(r s_1 -r_1 s) +(s_1 t -s t_1) y}{(r+sx+ty)^2} &
          \dfrac{(r t_1 -r_1 t) -(s_1 t -s t_1) x}{(r+sx+ty)^2} \\[0.5cm]
          \dfrac{(r s_2 -r_2 s) +(s_2 t -s t_2) y}{(r+sx+ty)^2} &
          \dfrac{(r t_2 -r_2 t) -(s_2 t -s t_2) x}{(r+sx+ty)^2}
      \end{pmatrix}.
\end{equation}

\begin{proposition}
  \label{prop-estimate-dpsi}
  Let $\psi_{i_1,\,\dots,\,i_n} : A\to C_{i_1,\,\dots,\,i_n}$ be a
  local inverse of $V^n$ and let $D\psi_{i_1,\,\dots,\,i_n}$ be its
  Jacobian matrix. Then there exists a sequence $(d(n))_{n\geq 0}$
  such that $\lim_{n\rightarrow \infty}d(n)=0$ and
  \[
      \max \left\{ \sup_A\left( \left|(D\psi_{i_1,\,\dots,\,i_n})_{11}
          \right| + \left|(D\psi_{i_1,\,\dots,\,i_n})_{21}\right|
        \right) ,\, \sup_A \left( \left|
            (D\psi_{i_1,\,\dots,\,i_n})_{12} \right| + \left|
            (D\psi_{i_1,\,\dots,\,i_n})_{22}\right| \right) \right\}
      \le d(n).
  \]
\end{proposition}

\begin{proposition}\label{prop-det}
    Let $\psi_i:A \to \R^2$ be a local inverse of $V$ given by
    $\psi = \phi_{\sigma_1} \circ \phi_{\sigma_2} \circ \cdots \circ
    \phi_{\sigma_k}$. Then $r+s+t>0$ and
    \[
        J\psi_i(x,y) = \frac{1}{(r+sx+ty)^3}.
    \]
\end{proposition}

\noindent Proposition \ref{prop-estimate-dpsi} follows from results used to prove the main result in \cite{nog}. We give a proof of the proposition in Appendix~\ref{sec:matrices} for completeness and also because we obtain an explicit estimate for the sequence $d(n)$. Concerning Proposition \ref{prop-det}, it is immediate from the construction of the local inverses of $V$ that $r+s+t>0$. The formula for the Jacobian determinant is a particular case of a result in \cite{veech} (see also Proposition 2 in \cite{schw-book}). We are now in a position to prove that
conditions (h3), (h5), (h7) and (h8) hold for our system.

\vskip 0.2cm
\begin{proof}[Proof of (h3)]
    Using Proposition~\ref{prop-estimate-dpsi} we have
    \begin{align*}
        \left\|\psi_{i_1,\,\ldots,\,i_n}(x_1,y_1) -
          \psi_{i_1,\,\ldots,\,i_n}
          (x_2,y_2)\right\| &\le\\
        &\hspace{-3cm}\le\sup_A\left( \left|
            (D\psi_{i_1,\,\ldots,\,i_n})_{11} \right| + \left|
            (D\psi_{i_1,\,\ldots,\,i_n})_{21}\right| \right)|x_1-x_2|+\\
        &\hspace{0cm}+ \sup_A\left(
          \left|(D\psi_{i_1,\,\ldots,\,i_n})_{12}\right| +
          \left|(D\psi_{i_1,\,\ldots,\,i_n})_{22}\right|\right)|y_1-y_2| \le\\
        &\hspace{-3cm}\leq d(n)\left(|x_1-x_2|+|y_1-y_2|\right) \le
        \sqrt{2}\,d(n)\|(x_1,y_1) - (x_2,y_2)\|.
    \end{align*}
    Since $C_{i_1,\,\ldots,\,i_n} = \psi_{i_1,\,\ldots,\,i_n}(A)$, we
    have
    \[
        \diam C_{i_1,\,\ldots,\,i_n}\le \sqrt{2}d(n)\cdot \diam A =
        \frac{\sqrt{10}}{3}d(n),
    \]
    so that (h3) is satisfied with
    $\sigma(n) = \frac{\sqrt{10}}{3}d(n)$.
\end{proof}

\begin{proof}[Proof of (h8)]
    The above proof of (h3) also shows that (h8) is satisfied with
    $r_2= \frac{3}{\sqrt{5}}$.
\end{proof}

\begin{proof}[Proof of (h5)]
    We have $V^n(C_{i_1,\,\ldots,\,i_n})=A$ for all $n$ and all
    cylinders $C_{i_1,\,\ldots,\,i_n}$. For all $(x,y)\in A$ holds
    \[
        \frac{1}{27}(r+s+t)^3 \le (r+sx+ty)^3 \le (r+s+t)^3
    \]
    if the coefficients $r$, $s$, and $t$ are non-negative. Then from
    Proposition~\ref{prop-det} it follows that condition (h5) holds
    with $\lambda=27$.
\end{proof}

\begin{proof}[Proof of (h7)]
    Using Proposition~\ref{prop-det} we have
    \[
        \left|J\psi_{i_1,\,\ldots,\,i_n} (x_1,y_1) -
          J\psi_{i_1,\,\ldots,\,i_n} (x_2,y_2) \right| \le 3\sqrt{2}\,
        \left( \max_{(x,y)\in A} \frac{s+t}{(r+sx+ty)^4}\right)
        \|(x_1,y_1)-(x_2,y_2)\|\, .
    \]
    Arguing as above
    \[
        \max_{(x,y)\in A} \frac{s+t}{(r+sx+ty)^4} \le 3\,
        \frac{s+t}{r+s+t}\max_{(x,y)\in A}
        J\psi_{i_1,\,\ldots,\,i_n}(x,y) \le 81 \inf_{(x,y)\in A}
        J\psi_{i_1,\,\ldots,\,i_n}(x,y)\le 81
        m(C_{i_1,\,\ldots,\,i_n})
    \]
    where $m$ denotes the normalised Lebesgue measure on $A$. It
    follows that (h7) holds with $r_1 = 243\sqrt{2}$.
\end{proof}

\noindent We have thus proved that the induced map $V$ of $S$ on the
triangle $A$ satisfies the assumptions of
Proposition~\ref{prop:fs-cfm}, hence the induced map $V$ is
$\psi$-mixing. As a consequence, our system $(\trianglecl,\mu,S)$
satisfies the assumptions of Proposition~\ref{prop:cfm-pde}, hence it
is pointwise dual ergodic.

The second part of Theorem~\ref{thm-pde} concerns the return sequence
$a_n(S)$. To achieve the conclusion, we use \cite[Lemma
3.7.4]{aaronson:iet} and \cite[Proposition 7]{zwei}: these results
imply that, given the \emph{wandering rate} $w_n(A)$ of the set $A$, it holds
\[
a_n(S) \asymp \frac{n}{w_n(A)}
\]
 In
Appendix~\ref{sec:varying-seq} we recall the definition of the
wandering rate $w_n(A)$ of the set $A$ and show that
$(w_n(A))_{n\geq 1}$ satisfies $w_n(A)\asymp \log^2 n$ (see Propositions  \ref{prop:w-above} and \ref{prop:w-below}).
This completes the proof of Theorem~\ref{thm-pde}.

\begin{remark} \label{rem:slowly}
It is known (see e.g. \cite[Lemma
3.7.4]{aaronson:iet} and \cite[Proposition 7]{zwei}) that if $w_n(A)$ is regularly varying of index $1-\alpha$, then
\[
    a_n(S) \sim \frac{1}{\Gamma(2-\alpha)\Gamma(1+\alpha)}
    \frac{n}{w_n(A)},
\]
and $a_n(S)$ is a regularly varying sequence of index $\alpha$. Hence to obtain that $a_n(S)$ is regularly varying of index $\alpha=1$ it is enough to show that $w_n(A)$ is slowly varying. The last is the additional assumption we need in Theorem \ref{thm-wlln} and Corollary \ref{cor-kuzdig}.Unfortunately we don't have a proof that $w_n(A)$ is slowly varying. At the end of Appendix~\ref{sec:varying-seq} we discuss this property.
\end{remark}

\section{Pointwise convergence of Birkhoff averages for a class of
  non-summable observables}
\label{sec:ba}

As already mentioned at the beginning of Section~\ref{sec:wlln}, for
infinite-measure-preserving systems (like our map $S$, or equivalently
the map $F$ on the strip as described in Section~\ref{sec:ess}), the
strict analogue of Birkhoff's Ergodic Theorem is trivial, in the sense
that it tells us only that for every observable $f\in L^1(\mu)$, the
Birkhoff averages of $f$ for a system $(X,\mu,R)$
\[
  \frac 1n \sum_{k=0}^{n-1}f\circ R^k(x)
\]
converge $\mu$-almost everywhere to zero. In a recent paper \cite{LM},
the question of convergence of Birkhoff sums for ``global
observables'', which were first introduced by Lenci \cite{limix, lpmu}
in the context of infinite mixing, is considered. In \cite{LM}, a
global observable is rather vaguely defined to be any $L^\infty$
function for which a Birkhoff-like theorem could in principle be shown
to hold. We would like to apply one of the results of this paper to
give certain examples of $L^\infty$ observables for our map $F$ for
which the Birkhoff average can shown to be almost everywhere
constant. In order to state this result, first we need to recall a
certain dynamically-defined partition.

Assume that $(X, \BB, \mu, R)$ is a conservative and ergodic
system. Given a set $L_0$ with $0 < \mu(L_0) < +\infty$, we have that
\begin{equation*}
    \bigcup_{k\ge 0}R^{-k} L_0 = X \pmod{\mu},
\end{equation*}
that is, $L_0$ is a sweep-out set. Now recursively define, for each
$k \ge 1$,
\[
  \label{lk}
  L_k \coloneqq \left( R^{-1} L_{k-1}\right) \setminus L_0.
\]
Then the collection $\{ L_k \}_{k\geq 0}$ forms a partition of $X$.

\begin{theorem}[\cite{LM}]
  \label{thmLM}
  Let $(X, \BB, \mu, R)$ be an infinite-measure-preserving,
  conservative, ergodic dynamical system, endowed with the partition
  $\{L_k\}_{k\geq 0}$, as described above. Let $f\in L^\infty(X, \mu)$
  admit $f^*\in \C$ with the following property:
  $\forall \varepsilon>0$, $\exists N, K\in \N$ such that
  $\forall x \in \bigcup_{k \ge K} L_k$,
  \[
      \left| \frac1N \sum_{k=0}^{N-1} f\circ R^k(x)-f^* \right| \le
      \varepsilon.
  \]
  Then for $\mu$-almost every $x \in X$,
  \[
      \lim_{n\rightarrow+\infty}\frac 1n \sum_{k=0}^{n-1}f\circ R^k(x)
      = f^*.
  \]
\end{theorem}

Let us recall the system $(\Sigma, \rho, F)$ defined in Section
\ref{sec:ess}, which is isomorphic to $(\trianglecl,\mu,S)$, and let
us describe a suitable partition of the strip $\Sigma$ for the
application of Theorem \ref{thmLM}. We let $L_0\coloneqq \Pi_0$, and
then, as above, define recursively the sets
\[
    L_k\coloneqq \left(F^{-k}L_{k-1}\right)\setminus L_0 = \{(u, v)\in
    \Sigma \,:\, k\le v< k+1\} = \Sigma_k.
\]
As a class of observables we consider the set $\mathcal G$ of
functions $f(u,v) \coloneqq g(u)\cdot h(v)$, where
$g:(0,1)\rightarrow \R$ is a bounded function and $h:\R\rightarrow \R$
is a continuous $\alpha$-periodic function, with
$\alpha\in \R\setminus \Q$.

\begin{theorem}\label{thm:BT}
    Let $(\Sigma, \rho, F)$ be the system defined in Section
    \ref{sec:ess}, and $f:\Sigma \rightarrow \C$ a function
    $f(u,v)=g(u)h(v)$ in the space $\mathcal G$ defined above, with
    $g$ constant if $\int_0^\alpha h dv\not= 0$. Then there exists a
    constant $f^*\in \C$ such that for $\rho$-almost every
    $(u,v)\in \Sigma$
    \[
        \lim_{n\rightarrow+\infty}\frac1n\sum_{k=0}^{n-1}f\circ F^k(u,
        v)=f^*.
    \]
\end{theorem}
\begin{proof}
    For all $N\in \N$ and for all $(u,v)\in \Sigma$ with $v > N$ we
    have
    \begin{align*}
        \frac1N\sum_{k=0}^{N-1}f\circ F^k(u, v) &=
        \frac1N\sum_{k=0}^{N-1}f(u, v-k) =\\
        &=\frac1N g(u)\sum_{k=0}^{N-1} h(v-k) =
        g(u) \cdot \frac1N \sum_{k=0}^{N-1} h\circ \tau^k(v),
    \end{align*}
    where $\tau:\R/\alpha\Z \rightarrow \R/\alpha\Z$ is defined by
    $\tau(x) \coloneqq x-1\pmod{\alpha}$. Note that $\tau$ preserves
    the Lebesgue measure and is topologically conjugate to the
    rotation $R_{\frac 1\alpha}:\R/\Z\rightarrow \R/\Z$,
    $R_{\frac 1\alpha}(x) \coloneqq x-\frac 1\alpha \pmod{1}$. The map
    $R_{\frac 1\alpha}$ is uniquely ergodic with respect to the
    Lebesgue measure since $\alpha\in \R\setminus\Z$, and thus so is
    $\tau$. It then follows, since $h$ is continuous and $\R/\alpha\Z$
    is compact, that
    \[
    \lim_{n\rightarrow+\infty} \frac1n \sum_{k=0}^{n-1} h\circ
    \tau^k(v) = \frac{1}{\alpha} \int_0^\alpha h(v)\,dv \eqqcolon h^*
\]
uniformly on $\R$. If $h^*=0$, there exists $N_\star\in \N$ such that
for all $n\geq N_\star$
\[
    \left| \frac1n \sum_{k=0}^{n-1} h\circ \tau^k(v) \right| <
    \frac{\varepsilon}{\norm{g}_\infty}.
\]
If we choose $N=N_\star$, $K>N$, and $f^*=0$, then for all $(u,v)\in
\Sigma$ with $v\geq K$
\[
    \left| \frac1N\sum_{k=0}^{N-1}f\circ F^k(u, v) - f^* \right| =
    \left| \frac1N\sum_{k=0}^{N-1} g(u) h(v-k) \right| \leq
    \norm{g}_{\infty} \left| \frac1N\sum_{k=0}^{N-1} h\circ
    \tau^k(v)\right| < \varepsilon,
\]
and we can apply Theorem \ref{thmLM} with $f^*=0$. In the case
$h^*\not= 0$, we can repeat the argument when $g(u)\equiv \bar g$ is a
constant function. In that case there exists $N_\star\in \N$ such that
for all $n\geq N_\star$
\[
    \left| \frac1n \sum_{k=0}^{n-1} h\circ \tau^k(v) - h^*\right| <
    \frac{\varepsilon}{\bar g}\, ,
\]
and we can apply Theorem \ref{thmLM} as above with $f^*=\bar g\, h^*$.
\end{proof}

\section{A complete triangular tree of rational pairs}
\label{sec:tree}


Our aim in this section is to construct a tree that contains every
pair of rational numbers in $\Q^2\cap \trianglecl$, first by using a
modified version of the map $S$ and then, equivalently, by giving a
geometric contruction by way of a mediant operation defined on pairs
of rational numbers. The construction mimics that of the Farey tree,
generated by the Farey map, which we now recall.

Firstly, the \emph{Farey map} is the map $F:[0,1]\rightarrow [0,1]$
defined by setting
\[
    F(x) \coloneqq
    \begin{cases}
        \frac{x}{1-x} & \text{if } 0\leq x\leq \frac 12\\[0.2cm]
        \frac{1-x}{x} & \text{if } \frac 12 \leq x \leq 1\\
    \end{cases}
\]
We generate a binary tree using the map $F$ by defining the levels
$\mathcal{L}_n\coloneqq F^{-n}\left(\frac12\right)$, with the vertices
connected as shown in Figure~\ref{fig:Farey}.  Note that the two
``children'' of each vertex are not simply the inverse images of that
vertex. If we label a step down to the left with ``0'' and a step down
to the right with ``1'', the position of a given rational number
$\frac pq \in \mathcal{L}_n$ is described by a path
$\omega_1\cdots\omega_n$, with each $\omega_i\in\{0, 1\}$, and such
that if $F_0$ and $F_1$ denote the inverse branches of $F$, then
$\frac pq= F_{\omega_1}\circ\cdots\circ F_{\omega_n}\left(\frac
  12\right)$. One of the most important properties of the Farey tree
is that it contains all the rational numbers in the interval $(0,1)$,
and each rational number appears in the tree exactly once\footnote{In
  particular, a rational number $\frac pq\in \mathcal{L}_n$ if and
  only if its continued fraction expansion
  $\frac pq = [a_1,\,\ldots,\,a_r]$ with $a_r>1$ is such that
  $\sum_{i=1}^r a_i=n+2$.}. In other words,
$\bigcup_{k=0}^{+\infty} F^{-k}\left(\frac 12\right) = \Q\cap (0,1)$.

\begin{figure}[h]
    \begin{tikzpicture}[
    level 1/.style = {sibling distance=4cm},
    level 2/.style = {sibling distance=2cm},
    level 3/.style = {sibling distance=1cm},
    level distance          = 1cm,
    edge from parent/.style = {draw},
    scale=1.2
    ]

    \foreach \n in {0,1,2,3} {
      \pgfmathsetmacro\p{\n}
      \node at (5,-\n) {$\mathcal{L}_{\pgfmathprintnumber\p}$};
    }

    \node {$\frac 12$}
    child{
      node {$\frac 13$}
      child{
        node {$\frac 14$}
        child{
          node {$\frac 15$}
        }
        child{
          node {$\frac 27$}
        }
      }
      child{
        node {$\frac 25$}
        child{
          node {$\frac 38$}
        }
        child{
          node {$\frac 37$}
        }
      }
    }
    child{
      node {$\frac 23$}
      child{
        node {$\frac 35$}
        child{
          node {$\frac 47$}
        }
        child{
          node {$\frac 58$}
        }
      }
      child{
        node {$\frac 34$}
        child{
          node {$\frac 57$}
        }
        child{
          node {$\frac 45$}
        }
      }
    };
\end{tikzpicture}
    \caption{The first four levels of the Farey tree.}\label{fig:Farey}
\end{figure}
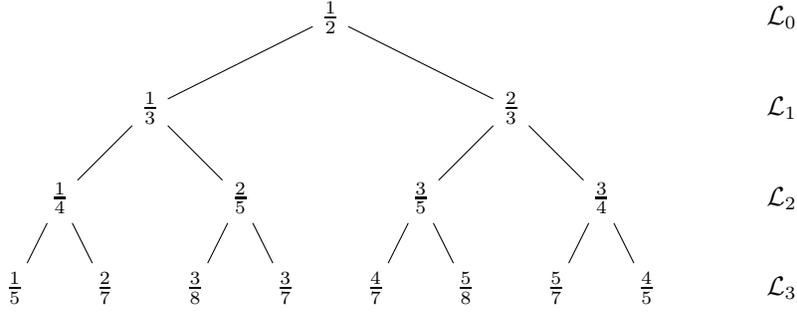

Another way to define the levels of the
Farey tree is by considering the Stern-Brocot sets
$(\mathcal{F}_n)_{n\geq -1}$, where we define
$\mathcal{F}_{-1}\coloneqq\left\{\frac01, \frac11\right\}$, and for
all $n\geq 0$, $\mathcal{F}_n$ is defined recursively from
$\mathcal{F}_{n-1}$ by inserting the mediant of each pair of
neighbouring fractions. Recall that the mediant of two fractions
$\frac pq$ and $\frac rs$ is defined to be
\[
    \frac pq \oplus \frac rs \coloneqq \frac{p+r}{q+s}.
\]
It is easy to verify that the mediant falls between the two rational
numbers it is computed from, that is if $\frac pq<\frac rs$ then
$\frac pq < \frac pq \oplus \frac rs < \frac rs$. The first few of the
Stern-Brocot sets are as follows:
\[
    \mathcal{F}_0 = \set{\frac 01, \frac 12, \frac 11},\quad
    \mathcal{F}_1 = \set{\frac 01, \frac 13, \frac 12, \frac 23, \frac
      11},\quad \mathcal{F}_2 = \set{\frac 01, \frac 14, \frac 13,
      \frac 25, \frac 12, \frac 35, \frac 23, \frac 34, \frac 11}.
\]
It is also straightforward to prove that
$\# \mathcal{F}_n = 2^{n+1}+1$, and that
$\mathcal{L}_n=\mathcal{F}_n\setminus\mathcal{F}_{n-1}$ for all
$n\geq 0$. For more details on the Farey tree, we refer to \cite{BI}.

We now describe the construction of our two-dimensional Farey-like
tree. We use the local inverse
$\phi_0: \trianglecl\setminus \{x=y\} \rightarrow \Gamma_0$, the
restricted local inverse
$\phi_1 : \trianglecl\setminus \Lambda \rightarrow \Gamma_1\setminus
\Lambda$ (which we will continue to call $\phi_1$), and a new map
$\phi_2: \{x=y \,:\, 0\le x\le 1\} \to \Lambda$ defined to be
$\phi_2(x,x)\coloneqq (x,0)$. The geometric action of $\phi_0$ and
$\phi_1$ is shown in Figure~\ref{fig:phi_maps}. These three maps are
the local inverses of the map
\[
    \tilde S:\trianglecl \rightarrow \trianglecl,\quad \tilde S(x,y)
    \coloneqq
    \begin{cases}
        S(x,y) & \text{if }(x,y)\in\trianglecl\setminus
        \Lambda\\[0.2cm]
        (x,x) & \text{if } (x,y)\in \Lambda
    \end{cases}.
\]
The map $\tilde S$ is a modified version of the map $S$ defined in
\eqref{eq-slow}.

\begin{figure}[h]
    \begin{tikzpicture}[scale=2.5]
    \def\t{1.6}

    \draw[very thin,fill=gray!20] (0,0) -- (1,0) -- (1,1) -- cycle;
    \draw[very thick] (0,0) -- (1,0) -- (1,1);
    \draw[very thick,dashed] (0,0) -- (1,1);

    \draw[very thin] (\t,0) -- (\t+1,0) -- (\t+1,1) -- cycle;
    \draw[very thick,fill=gray!20] (\t+1,0) -- (\t+1,1) -- (\t+0.5,0.5);
    \draw[very thick,dashed] (\t+0.5,0.5) -- (\t+1,0);

    \draw[-latex'] (1.15,0.4) -- (\t-0.15,0.4);
    \node[above] at ({(\t+1)/2},0.4) {$\phi_0$};
\end{tikzpicture}
\hspace{1.5cm}
\begin{tikzpicture}[scale=2.5]
    \def\t{1.6}

    \draw[very thin,fill=gray!20] (0,0) -- (1,0) -- (1,1) -- cycle;
    \draw[very thick] (1,0) -- (1,1) -- (0,0);
    \draw[very thick,dashed] (0,0) -- (1,0);

    \draw[very thin] (\t,0) -- (\t+1,0) -- (\t+1,1) -- cycle;
    \draw[very thick,fill=gray!20] (\t+1,0) -- (\t+0.5,0.5) -- (\t,0);
    \draw[very thick,dashed] (\t,0) -- (\t+1,0);

    \draw[-latex'] (1.15,0.4) -- (\t-0.15,0.4);
    \node[above] at ({(\t+1)/2},0.4) {$\phi_1$};

\end{tikzpicture}
    \caption{Geometric action of the maps $\phi_0$ and
      $\phi_1$.}\label{fig:phi_maps}
\end{figure}
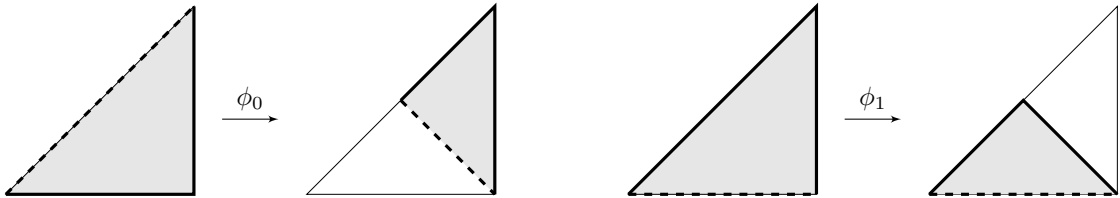

We now define the sequence $(\mathcal{T}_n)_{n\geq -1}$ of the levels
of the tree associated to $\tilde S$. First set
\[
    \mathcal{T}_{-1} \coloneqq \left\{(0,0),\,(1,0),\,(1,1)\right\}
    \quad\text{and}\quad
    \mathcal{T}_{0} \coloneqq \left\{ \singy{1}{2}{0},\,
      \singx{1}{1}{2},\, \triple{1}{1}{2} \right\},
\]
which include the vertices of the triangle $\triangle$ and the middle
points of the sides, respectively.

\begin{definition}
    For each $n\geq -1$, define
    $\mathcal{B}_n \coloneqq \mathcal{T}_n\cap \partial\triangle$ and
    $\mathcal{I}_n\coloneqq \mathcal{T}_n\cap
    \accentset{\circ}{\triangle}$ to be, respectively, the
    \emph{boundary points} and the \emph{interior points} of the
    $n$-th level of the tree. Moreover, we respectively denote with
    $\mathcal{B}_{\leq n}\coloneqq \bigcup_{k=-1}^{n} \mathcal{B}_k$
    and
    $\mathcal{I}_{\leq n}\coloneqq \bigcup_{k=-1}^{n} \mathcal{I}_k$
    the boundary and interior points of the tree up to level $n$.
\end{definition}

Clearly $\mathcal{B}_{-1}= \mathcal{T}_{-1}$ and
$\mathcal{B}_{0}= \mathcal{T}_{0}$. We now define precisely how the
levels of the tree are constructed, by showing all the possibilities
for taking counterimages depending on the location of the point in
$\trianglecl$. Let $n\geq 0$.

\begin{enumerate}[label={\upshape(R\arabic*)},wide = 0pt,leftmargin=*]
  \item\label{rule1} An interior point
    $\triple{p}{r}{q}\in \mathcal{I}_n$ generates the two interior
    points $\triple{q}{p}{r+q}$ and $\triple{p}{r}{r+q}$ in
    $\mathcal{I}_{n+1}$, through the application of $\phi_0$ and
    $\phi_1$, respectively.
  \item\label{rule2} A boundary point
    $\triple{p}{p}{q}\in \mathcal{B}_n$ generates the point
    $\singy{p}{q}{0}\in \mathcal{B}_n$ through the application of
    $\phi_2$ and the boundary point
    $\triple{p}{p}{p+q}\in \mathcal{B}_{n+1}$ through the application
    of $\phi_1$.
  \item\label{rule3} A boundary point
    $\singy{p}{q}{0}\in \mathcal{B}_n$ generates the point
    $\singx{1}{p}{q}\in \mathcal{B}_n$ through the application of
    $\phi_0$.
  \item\label{rule4} A boundary point
    $\singx{1}{p}{q}\in \mathcal{B}_n$ generates the boundary point
    $\triple{q}{q}{p+q}\in \mathcal{B}_{n+1}$ and the interior point
    $\triple{q}{p}{p+q}\in\mathcal{I}_{n+1}$, through the application
    of $\phi_0$ and $\phi_1$, respectively.
\end{enumerate}
The basic portions of the counterimages tree generated from a boundary
point and from an interior point are shown in
Figure~\ref{fig:rules}. Note that the points of the tree always have
rational coordinates, since we start from points with rational
coordinates and $\phi_0$, $\phi_1$ and $\phi_2$ are linear fractional
maps. Furthermore, taking a counterimage does not necessarily implies
that the level in the tree changes. Indeed, applying rules
\ref{rule1}, \ref{rule4}, and rule \ref{rule2} with $\phi_1$ makes the
level to increase, whereas applying the other rules does not change
the level. The levels $\mathcal{T}_0$, $\mathcal{T}_1$ and
$\mathcal{T}_2$ of the tree are shown in Figure~\ref{fig:tree} at the
end of the paper. Note that we always write the two fractions of each
pair reduced to their least common denominator, apart from the pairs
containing $0$ and/or $1$. This choice also has a geometric
motivation, as we shall remark after Definition~\ref{def:mediant}.

\begin{figure}[h]
    \begin{tikzpicture}[
    scale=0.95,
    level 2/.style = {sibling distance=2.25cm},
    level distance          = 2.5cm,
    edge from parent/.style = {draw,thin,-latex'},
    every node/.style       = {-latex'},
    ]

    \node at (-1.75,0) {$\mathcal{T}_n$};
    \node at (-1.75,-2.5) {$\mathcal{T}_{n+1}$};

    \node[inner sep=1pt,circle,fill=gray!20,right] at (0.1,-1.25)
    {\footnotesize R2};

    \node[inner sep=1pt,circle,fill=gray!20] at (8,-1.25)
    {\footnotesize R4};

    \node {$\triple{p}{p}{q}$} [grow=down]
    child [grow=down]{
      node {$\triple{p}{p}{p+q}$}
	  child [grow=right] {
          	node {$\singy{p}{p+q}{0}$}
            child [grow=right] {
              node {$\singx{1}{p}{p+q}$}
              edge from parent node[above] {$\phi_0$} node[inner sep=1pt,circle,fill=gray!20,below,yshift=-0.1cm] {\footnotesize R3}
			}
          	edge from parent node[above] {$\phi_2$} node[inner sep=1pt,circle,fill=gray!20,below,yshift=-0.1cm] {\footnotesize R2}
		  }
      edge from parent node[left] {$\phi_1$}
    }
    child[grow=right] {
      node at (1.5,0) {$\singy{p}{q}{0}$}
      child[grow=right] {node  at (1.5,0) {$\singx{1}{p}{q}$} [grow=down]
        child {
          node {$\triple{q}{p}{p+q}$}
          edge from parent node[left] {$\phi_1$}
        }
		child {
          node {$\triple{q}{q}{p+q}$}
		  child [grow=right] {
          	node {$\singy{q}{p+q}{0}$}
            child [grow=right] {
              node {$\singx{1}{q}{p+q}$}
              edge from parent node[above] {$\phi_0$} node[inner sep=1pt,circle,fill=gray!20,below,yshift=-0.1cm] {\footnotesize R3}
			}
          	edge from parent node[above] {$\phi_2$} node[inner sep=1pt,circle,fill=gray!20,below,yshift=-0.1cm] {\footnotesize R2}
		  }
          edge from parent node[right] {$\phi_0$}
        }
        edge from parent node[above] {$\phi_0$} node[inner sep=1pt,circle,fill=gray!20,below,yshift=-0.1cm] {\footnotesize R3}
      }
      {edge from parent node[above] {$\phi_2$} node[inner sep=1pt,circle,fill=gray!20,below,yshift=-0.1cm] {\footnotesize R2}}
    };
\end{tikzpicture}\\[0.75cm]
    \begin{tikzpicture}[
    scale=0.95,
    level 2/.style = {sibling distance=2.25cm},
    level 1/.style = {sibling distance=2.25cm},
    level distance          = 2.25cm,
    edge from parent/.style = {draw,-latex'},
    every node/.style       = {-latex'},
    ]

    \node at (-2.75,0) {$\mathcal{T}_n$};
    \node at (-2.75,-2.25) {$\mathcal{T}_{n+1}$};

    \node[inner sep=1pt,circle,fill=gray!20] at (0,-1.125)
    {\footnotesize R1};

    \node {$\triple{p}{r}{q}$} [grow=down]
    child {
      node {$\triple{q}{p}{r+q}$}
      edge from parent node[left] {$\phi_0$}
    }
    child {
      node {$\triple{p}{q}{r+q}$}
      edge from parent node[right] {$\phi_1$}
    };
\end{tikzpicture}
    \caption{\emph{Above}. Basic portion of the tree generated from a
      boundary point $\triple{p}{p}{q}\in
      \mathcal{B}_n$. \emph{Below}. Basic portion of the tree
      generated from an interior point
      $\triple{p}{r}{q}\in \mathcal{I}_n$.}\label{fig:rules}
\end{figure}
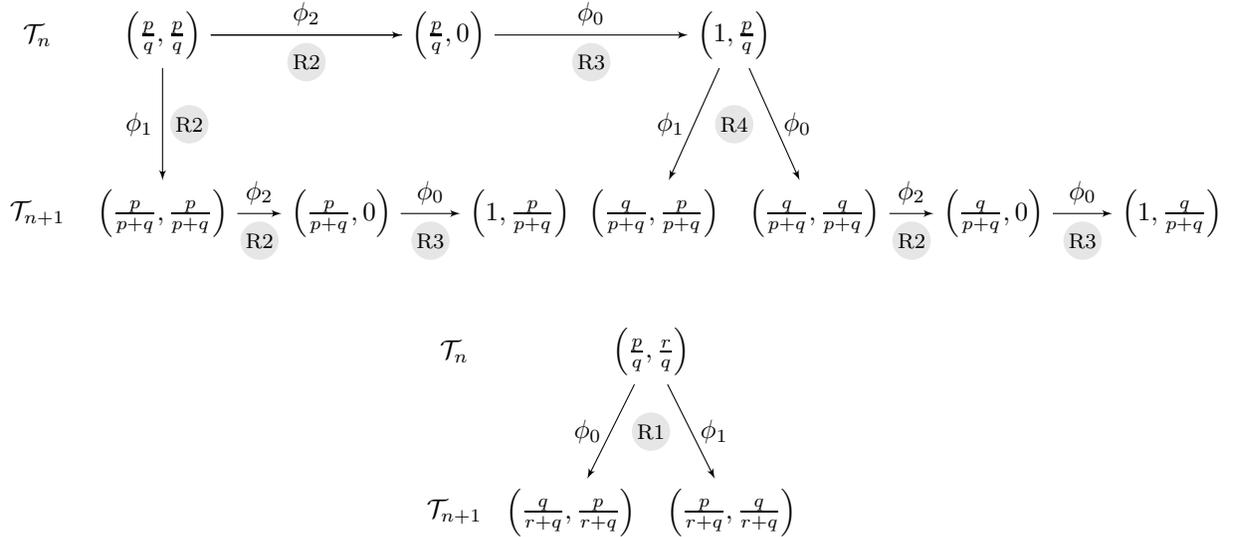

\begin{lemma}\label{lemma:count_Tn}
    For all $n\geq 0$ we have
    \[
        \#\mathcal{B}_n = 3\cdot 2^n\quad\text{and}\quad
        \#\mathcal{I}_n = n 2^{n-1},
    \]
    with the points of $\mathcal{B}_n$ equally distributed on the
    three sides of $\triangle$. As a consequence, the number of points
    of each level of the tree is given by
    \[
        \#\mathcal{T}_n = 3\cdot 2^n + n 2^{n-1}.
    \]
\end{lemma}
\begin{proof}
    We argue by induction on $n\geq 0$. If $n=0$ we have
    $\mathcal{B}_0 = \mathcal{T}_0$ and
    $\mathcal{I}_0=\emptyset$. Thus $\#\mathcal{B}_0=3$, with one
    point on each side of $\triangle$, and $\#\mathcal{I}_0=0$: the
    base case is proved. Suppose $\#\mathcal{B}_n = 3\cdot 2^n$, with
    $2^n$ points on each side of $\triangle$,
    $\#\mathcal{I}_n = n 2^{n-1}$ for some $n\geq 0$ and consider the
    subsequent level of the tree. The points of $\mathcal{B}_{n+1}$
    can be obtained from those of $\mathcal{B}_n$ as follows (refer to
    Figure~\ref{fig:rules}):
    \begin{itemize}
      \item $\mathcal{B}_n$ contains $2^n$ points on
        $\trianglecl\cap \{x=y\}$, each of which gives $3$ points in
        $\mathcal{B}_{n+1}$, one per side, applying \ref{rule2} with
        $\phi_1$, followed by \ref{rule2} with $\phi_2$, and
        \ref{rule3};
      \item $\mathcal{B}_n$ contains $2^n$ points on
        $\trianglecl\cap \{x=1\}$, each of which gives $3$ points in
        $\mathcal{B}_{n+1}$, one per side, applying \ref{rule4},
        followed by \ref{rule2} with $\phi_2$, and \ref{rule3}.
    \end{itemize}
    Note that the other $2^n$ points contained in $\mathcal{B}_n$ lie
    on the line $\{y=0\}$, and these points are all mapped back inside
    $\mathcal{B}_n$, in light of rule \ref{rule3}. Therefore,
    $\#\mathcal{B}_{n+1} = 2^n\cdot 3 + 2^n\cdot 3 = 3\cdot
    2^{n+1}$. Furthermore, by construction, we have $2^{n+1}$ points
    of $\mathcal{B}_{n+1}$ on each side of the triangle, so that the
    points of $\mathcal{B}_{n+1}$ are equally distributed on the three
    sides of $\triangle$.  The points of $\mathcal{I}_{n+1}$ are
    obtained from those of $\mathcal{T}_n$ in this way:
    \begin{itemize}
      \item each point in $\mathcal{I}_n$ generates two points in
        $\mathcal{I}_{n+1}$, according to rule \ref{rule1};
      \item each point in $\mathcal{B}_n\cap \{x=1\}$ gives one point
        in $\mathcal{I}_{n+1}$, using rule \ref{rule4}.
    \end{itemize}
    As a consequence
    $\#\mathcal{I}_{n+1} = n2^{n-1}\cdot 2 + 2^n = (n+1)2^n$. The
    inductive step is proved and the proof is complete.
\end{proof}

\begin{lemma}\label{lemma:char_Bn}
    Let $F:[0,1]\rightarrow [0,1]$ be the Farey map. For all $n\geq 0$
    we have
    \[
        \mathcal{B}_n = \set{\triple{p}{p}{q}, \singy{p}{q}{0},
          \singx{1}{p}{q} \,:\, \frac pq \in F^{-n}\left(\frac
            12\right)}.
    \]
\end{lemma}
\begin{proof}
    We claim that it is suffices to prove that for all $n\geq 0$
    \begin{equation}\label{Bn}
        \set{\triple{p}{p}{q} \,:\, \frac pq \in F^{-n}\left(\frac
            12\right)}\subseteq \mathcal{B}_n.
    \end{equation}
    Indeed by \ref{rule2} and \ref{rule3}, it easily follows that
    $\set{\triple{p}{p}{q}, \singy{p}{q}{0}, \singx{1}{p}{q} \,:\,
      \frac pq \in F^{-n}\left(\frac 12\right)}\subseteq
    \mathcal{B}_n$ for all $n\geq 0$, and since from
    Lemma~\ref{lemma:count_Tn} we have $\#\mathcal{B}_n=3\cdot 2^n$
    for all $n\geq 0$, the claim is proved. We now argue by induction
    to prove that \eqref{Bn} holds for all $n\geq 0$. If $n=0$ we have
    $\triple{1}{1}{2}\in \mathcal{B}_0$, thus the base case is
    proved. Now suppose that \eqref{Bn} holds for some $n\geq 0$ and let
    $\frac rs\in F^{-(n+1)}\left(\frac 12\right)$, so that
    $F\left(\frac rs\right)\in F^{-n}\left(\frac 12\right)$. In order
    to prove that $\triple{r}{r}{s}\in \mathcal{B}_{n+1}$ we
    distinguish between two cases.
    \begin{itemize}
      \item If $0\leq \frac rs \leq \frac 12$ then
        $F\left(\frac rs\right)=\frac r{s-r}$, and by the induction
        hypothesis $\triple{r}{r}{s-r}\in \mathcal{B}_n$. By rule
        \ref{rule2} we then have that
        $\phi_1\triple{r}{r}{s-r} = \triple{r}{r}{s}\in
        \mathcal{B}_{n+1}$.
      \item If $\frac 12 <\frac rs\leq 1$ then
        $F\left(\frac rs\right)=\frac {s-r}r$. By the induction
        hypothesis we have that
        $\triple{s-r}{s-r}{r}\in \mathcal{B}_n$ and, by construction,
        we also have that $\singx{1}{s-r}{r}\in \mathcal{B}_n$. Hence
        rule \ref{rule4} yields that
        $\phi_0\singx{1}{s-r}{r} = \triple{r}{r}{s}\in
        \mathcal{B}_{n+1}$.
    \end{itemize}
\end{proof}

We are now in a position to prove the first main result of this
section.

\begin{theorem} \label{thm:tree-compl}
    The tree defined by the level sets $\mathcal{T}_n$ is
    complete, that is,
    \[
        \bigcup_{n\geq -1} \mathcal{T}_n = \Q^2 \cap \trianglecl,
    \]
    and every pair of rational numbers appears in the tree
    exactly once.
\end{theorem}

\begin{proof}
    We have shown in Lemma~\ref{lemma:char_Bn} that every pair of
    rational numbers of the form $\singx{1}{a}{b}$, with
    $0\leq \frac ab\leq 1$, appears in some set $\mathcal{B}_n$. Let
    $\triple{p}{r}{q}$ be an arbitrary pair of rational numbers in the
    interior of the triangle $\triangle$. Then, as noted in Section
    \ref{sec:ts}, this point has a terminating triangle sequence, say
    $(\alpha_1,\,\ldots,\,\alpha_m)$. Then it follows that
    $S^{\sum_{i=1}^m \alpha_i+m-1}\triple{p}{r}{q} = \singx{1}{a}{b}$,
    for some rational number $\frac ab$. Thus, as a backward image of
    the point $\singx{1}{a}{b}$ under $S$, our original, arbitrary,
    point must lie in the tree. Moreover, that each point appears
    exactly once is clear, considering the geometric action of the
    maps $\phi_i$.
\end{proof}

At this point, we will begin the description of a tree of points from
the triangle using a mediant operation on pairs of rational
numbers. This will then be shown to be equivalent to the description
of our two-dimensional Farey tree given above in terms of
counterimages.

\begin{definition}
    \label{def:mediant}
    Let $\triple{p}{r}{q}$ and $\triple{p'}{r'}{q'}$ be two couples of
    fractions. We define their \emph{mediant} to be
    \[
        \triple{p}{r}{q} \oplus \triple{p'}{r'}{q'} \coloneqq
        \left(\frac pq \oplus \frac{p'}{q'},\frac rq \oplus
          \frac{r'}{q'}\right)= \couple{p+p'}{q+q'}{r+r'}{q+q'}.
    \]
\end{definition}

\noindent Note that we require that the two fractions of each couple
have the same denominator, in order that the mediant of two points
lies on the open segment joining the two points. In what follows we
always assume that the two fractions of each couple are reduced to
their least common denominator.

\begin{definition}\label{def:farey_sum}
    Let $\mathfrak{S}\subseteq \R^2$ be a line segment, and let
    $\mathfrak{R}\subseteq \mathfrak{S}\cap \Q^2$ be a finite subset
    of rational points on $\mathfrak{S}$. Let
    $r\coloneqq \#\mathfrak{R}$, with $2\leq r<+\infty$, and write
    $\mathfrak{R}=\{\mathfrak{r}_i \,:\, i=1,\,\ldots,\,r\}$ with
    $\mathfrak{r}_i\leq_{lex} \mathfrak{r}_{i+1}$ for all
    $i=1,\,\ldots,\,r-1$, where $\leq_{lex}$ is the lexicographic
    order on $\R^2$. We define the \emph{Farey sum of $\mathfrak{R}$}
    to be the set $\mathfrak{R}^\oplus$, where
    \[
        \mathfrak{R}^\oplus \coloneqq \{\mathfrak{r}_i\oplus
        \mathfrak{r}_{i+1} \,:\, i=1,\,\ldots,\,r-1\} \cup
        \mathfrak{R}.
    \]
\end{definition}

To simplify the presentation, in what follows the maps $\phi_0$ and
$\phi_1$ are extended to $\trianglecl$.

\begin{lemma}\label{lemma:preserve_Fsum}
    The maps $\phi_0$ and $\phi_1$ preserve the mediant of any two
    rational pairs in their respective domains.
\end{lemma}
\begin{proof}
    Let $\triple{p}{r}{q}$ and $\triple{p'}{r'}{q'}$ be two rational
    points in the domain of $\phi_0$. Then
    \[
        \phi_0\left(\triple{p}{r}{q}\oplus \triple{p'}{r'}{q'}\right) =
        \phi_0 \triple{p+p'}{r+r'}{q+q'} = \triple{q+q'}{p+p'}{r+q+r'+q'}
    \]
    and
    \[
        \phi_0\triple{p}{r}{q}\oplus \phi_0\triple{p'}{r'}{q'} =
        \triple{q}{p}{r+q} \oplus \triple{q'}{p'}{r'+q'} =
        \triple{q+q'}{p+p'}{r+q+r'+q'}.
    \]
    An analogous computation can be done for the map $\phi_1$: we
    leave the details to the reader.
\end{proof}

\noindent Let us observe here that, since $\phi_0$ and $\phi_1$
preserve the mediant operation and are also monotonic along line
segments (with respect to the lexicographic order), we have that for
$i=0,1$,
\begin{align}\label{eq:f_Roplus}
    \begin{split}
        \phi_i(\mathfrak{R})^\oplus &=
        \set{\phi_i(\mathfrak{r}_i)\oplus \phi_i(\mathfrak{r}_{i+1})
          \,:\, i=1,\,\ldots,\,r-1}\cup
        \phi_i(\mathfrak{R})=\\
        &= \set{ \phi_i(\mathfrak{r}_i\oplus \mathfrak{r}_{i+1}) \,:\,
          i=1,\,\ldots,\,r-1}\cup \phi_i(\mathfrak{R}) =
        \phi_i\left(\mathfrak{R}^\oplus\right).
    \end{split}
\end{align}

In order to give the definition of a tree of mediants, we first define
a sequence of measurable partitions $(\mathscr{P}_n)_{n\geq 0}$ of
$\trianglecl$, such that $\mathscr{P}_n$ consists of $2^n$
subtriangles of $\trianglecl$ and each $\mathscr{P}_n$ is a refinement
of the previous $\mathscr{P}_{n-1}$. Let $\mathscr{P}_0$ be the whole
triangle $\trianglecl$. The three vertices of $\trianglecl$ are
labelled with ``$0$'', ``$1$'' and ``$2$'' as follows:
\[
    v_0=(0,0)=\triple{0}{0}{1},\quad v_1=(1,0)=\triple{1}{0}{1},\quad
    v_2=(1,1)=\triple{1}{1}{1}.
\]
Taking the Farey sum between $v_0$ and $v_2$ one obtains
\[
    v_0\oplus v_2 = \triple{1}{1}{2}.
\]
We partition the triangle $\trianglecl$ into two subtriangles by the
line segment joining $v_1$ and $v_0\oplus v_2$. This determines the
partition $\mathscr{P}_1$. Moreover, we label the vertices of the two
subtriangles according to the geometric rule shown in
Figure~\ref{fig:relabel}: that is, the new vertex is labelled ``2'' in
both the subtriangles, the other vertices of the subtriangle
containing the old vertex ``0'' remain as they were, whereas in the
subtriangle containing the old vertex ``2'', this ``2'' becomes a
``1'' and the remaining vertex is labelled ``0'' (note that, in this
second subtriangle, this can be seen as a rotation of the old
labels). We now proceed inductively. Suppose we have the partition
$\mathscr{P}_n$, consisting of $2^n$ triangles. Each triangle of
$\mathscr{P}_n$ is partitioned into two subtriangles by the line
segment joining the vertex labelled ``1'' with the mediant of the
vertex ``0'' and the vertex ``2''. This gives us the next partition
$\mathscr{P}_{n+1}$. Figure~\ref{fig:partition_012} shows the
partitions $\mathscr{P}_0$, $\mathscr{P}_1$, and $\mathscr{P}_2$.

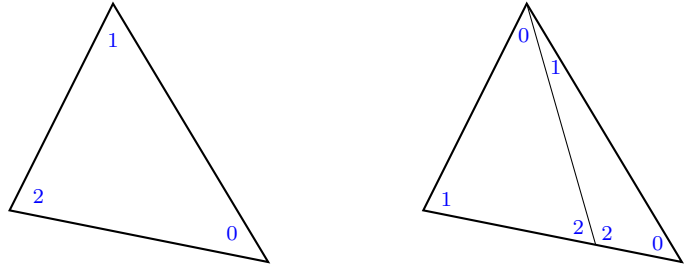
\begin{figure}[h]
    \begin{tikzpicture}[scale=2.75]
    \def\t{2}
    \def\a{0.175}
    \def\b{0.07}
    \def\c{0.1}
    \def\f{0.8}
    \draw[thick] (0,0) -- (1/2,1) -- (5/4,-1/4) -- cycle;
    \node[blue] at ($(5/4,-1/4)+(-2.5*\b,2*\b)$) {\footnotesize $0$};
    \node[blue] at ($(1/2,1)+(0,-\a)$) {\footnotesize $1$};
    \node[blue] at ($(0,0)+(4/5*\a,\b)$) {\footnotesize $2$};

    \draw[thick] ($(\t,0)+(0,0)$) -- ($(\t,0)+(1/2,1)$) -- ($(\t,0)+(5/4,-1/4)$) -- cycle;
    \draw[thin] ($(\t,0)+(1/2,1)$) -- ($(\t,0)+(5/6,-1/6)$);

    \node[blue] at ($(\t,0)+(5/4,-1/4)+(-5/3*\b,4/3*\b)$) {\footnotesize $0$};
    \node[blue] at ($(\t,0)+(1/2,1)+(0.8*\a,-1.75*\a)$) {\footnotesize $1$};
    \node[blue] at ($(\t,0)+(5/6,-1/6)+(8/25*\a,4/5*\b)$) {\footnotesize $2$};
    \node[blue] at ($(\t,0)+(1/2,1)+(-\f*1/4*\b,-\f*1.1*\a)$) {\footnotesize $0$};
    \node[blue] at ($(\t,0)+(0,0)+(\f*4/5*\a,\f*\b)$) {\footnotesize $1$};
    \node[blue] at ($(\t,0)+(5/6,-1/6)+(-\f*3/5*\a,\f*3/2*\b)$) {\footnotesize $2$};

\end{tikzpicture}
    \caption{Partition of a triangle of $\mathscr{P}_n$ into two
      subtriangles and relabelling of the
      vertices.}\label{fig:relabel}
\end{figure}

\begin{figure}[h]
    \begin{tikzpicture}[scale=3.7]
    \def\t{1.5}
    \def\a{0.175}
    \def\b{0.07}
    \def\c{0.1}
    \def\f{0.75}
    \draw[thick] (0,0) -- (1,0) -- (1,1) -- cycle;
    \node[blue] at ($(0,0)+(\a,\b)$) {\footnotesize $0$};
    \node[blue] at ($(1,0)+(-\b,\b)$) {\footnotesize $1$};
    \node[blue] at ($(1,1)+(-\b,-\a)$) {\footnotesize $2$};

    \draw[thick] (\t,0) -- (\t+1,0) -- (\t+1,1) -- cycle;
    \draw[thin] (\t+1/2,1/2) -- (\t+1,0);

    \node[blue] at ($(\t,0)+(\a,\b)$) {\footnotesize $0$};
    \node[blue] at ($(\t,0)+(1,0)+(-\a,\b)$) {\footnotesize $1$};
    \node[blue] at ($(\t,0)+(1/2,1/2)+(0,-\c)$) {\footnotesize $2$};
    \node[blue] at ($(\t,0)+(1,0)+(-\b,\a)$) {\footnotesize $0$};
    \node[blue] at ($(\t,0)+(1,1)+(-\b,-\a)$) {\footnotesize $1$};
    \node[blue] at ($(\t,0)+(1/2,1/2)+(\c,0)$) {\footnotesize $2$};

    \draw[thick] (2*\t,0) -- (2*\t+1,0) -- (2*\t+1,1) -- cycle;
    \draw[thin] (2*\t+1/2,1/2) -- (2*\t+1,0);
    \draw[thin] (2*\t+1,1) -- (2*\t+2/3,1/3);
    \draw[thin] (2*\t+1,0) -- (2*\t+1/3,1/3);

    \node[blue] at ($(2*\t,0)+(0,0)+(\f*\a,\f*\b)$) {\footnotesize $0$};
    \node[blue] at ($(2*\t,0)+(1,0)+(-\f*1.75*\a,\f*\b)$) {\footnotesize $1$};
    \node[blue] at ($(2*\t,0)+(1/3,1/3)+(0,-\f*\c)$) {\footnotesize $2$};
    \node[blue] at ($(2*\t,0)+(1,0)+(-2.75*\b,4/5*\a)$) {\footnotesize $0$};
    \node[blue] at ($(2*\t,0)+(1/2,1/2)+(0,-\f*\c)$) {\footnotesize $1$};
    \node[blue] at ($(2*\t,0)+(1/3,1/3)+(\f*\c,1/4*\b)$) {\footnotesize $2$};
    \node[blue] at ($(2*\t,0)+(1,1)+(-4/5*\a,-2.75*\b)$) {\footnotesize $0$};
    \node[blue] at ($(2*\t,0)+(1/2,1/2)+(\f*\c,0)$) {\footnotesize $1$};
    \node[blue] at ($(2*\t,0)+(2/3,1/3)+(-1/4*\b,\f*\c)$) {\footnotesize $2$};
    \node[blue] at ($(2*\t,0)+(1,0)+(-\f*\b,\f*\a)$) {\footnotesize $0$};
    \node[blue] at ($(2*\t,0)+(1,1)+(-\f*\b,-\f*1.75*\a)$) {\footnotesize $1$};
    \node[blue] at ($(2*\t,0)+(2/3,1/3)+(\f*\c,0)$) {\footnotesize $2$};
\end{tikzpicture}
    \caption{From left to right: partitions $\mathscr{P}_0$,
      $\mathscr{P}_1$, and $\mathscr{P}_2$, along with the labelling
      of the vertices.}\label{fig:partition_012}
\end{figure}
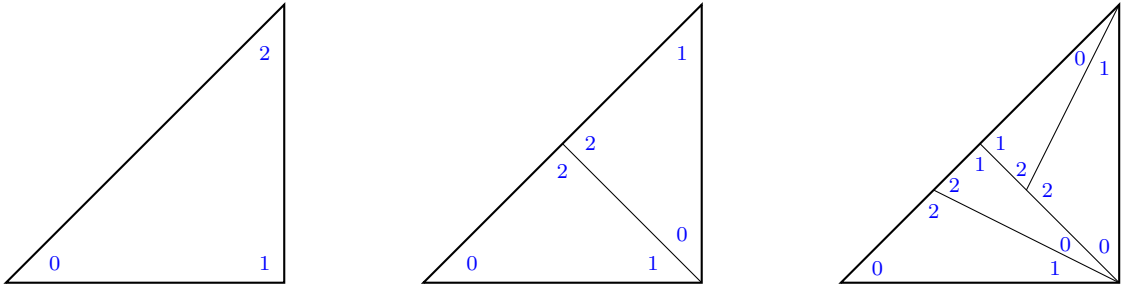

\noindent For the definition of the tree, in a Farey-like way, we
recursively define a sequence $(\mathcal{S}_n)_{n\geq -1}$ of nested
sets of pairs of rationals. First set
$\mathcal{S}_{-1} \coloneqq \left\{(0,0),\,(1,0),\,(1,1)\right\}$ and
consider the partition $\mathscr{P}_0$. The basic idea is to insert
the mediant of each pair of neighbouring points along each side of the
partition $\mathscr{P}_0$. For instance, at the first step we have
three sides (the sides of $\trianglecl$), and we add to
$\mathcal{S}_{-1}$ one point along each side, providing
\[
    \mathcal{S}_0 = \set{(0,0),\singy{1}{2}{0},(1,0),\singx{1}{1}{2},
      (1,1),\triple{1}{1}{2}}.
\]
Again we proceed inductively. Suppose we have the set $\mathcal{S}_n$
and consider the partition $\mathscr{P}_{n+1}$. The set
$\mathcal{S}_{n+1}$ is obtained from $\mathcal{S}_n$ inserting the
mediant of each pair of neighbouring points along each side of the
partition $\mathscr{P}_{n+1}$. In other words, for $n\geq -1$,
\[
    \mathcal{S}_{n+1} \coloneqq \bigcup_{\mathfrak{S}\in
      \mathscr{S}_{n+1}} (\mathfrak{S}\cap \mathcal{S}_n)^\oplus,
\]
where $\mathscr{S}_0$ is the set of the three sides of $\trianglecl$,
and $\mathscr{S}_n$ is obtained from $\mathscr{S}_{n-1}$ by adding the
line segments used to partition the triangles of $\mathscr{P}_{n-1}$
to obtain $\mathscr{P}_n$.

Now we will start to work towards showing that our two trees are in
fact identical. Figure \ref{fig:triangle_levels} at the end of the
section may ease the understanding of the argument. Let $n\geq 0$ and
let $\omega\in \{0,1\}^n$ be a word of length $|\omega|=n$ over the
two symbols ``$0$'' and ``$1$''. We define
\[
    \phi_\omega\coloneqq \phi_{\omega_1}\circ \phi_{\omega_2}\circ
    \cdots \circ \phi_{\omega_n}
    \quad\text{and}\quad
    \triangle_\omega\coloneqq \phi_\omega(\trianglecl).
\]

\begin{lemma}\label{lemma:partition-sides}
    Let $\ell$ be the open line segment joining $(1,0)$ and
    $\couple{1}{2}{1}{2}$, that is
    $\ell\coloneqq \set{(x,1-x) \,:\, \frac 12 <x<1}$. For $n\geq 0$
    the following holds.
    \begin{enumerate}[label={\upshape(\roman*)},wide = 0pt,leftmargin=*]
      \item $\mathscr{P}_n = \{\triangle_\omega\,:\, |\omega|=n\}$,
        and the labelling of the vertices of each $\triangle_\omega$
        is such that the vertex ``$k$'' of
        $\triangle_\omega=\phi_\omega(\trianglecl)$ is
        $\phi_\omega(v_k)$, for $k=0,1,2$.
      \item
        $\mathscr{S}_{n} = \mathscr{S}_0\cup
        \set{\widebar{\phi_\omega(\ell)} \,:\, |\omega|\leq n-1}$.
    \end{enumerate}
\end{lemma}
\begin{proof}
    (i) The statement is trivially true when $n=0$. We argue by
    induction on $n\geq 1$. From the definition of the maps $\phi_0$
    and $\phi_1$ it is straightforward to see that
    \[
        \mathscr{P}_1 = \set{\phi_0(\trianglecl),\phi_1(\trianglecl)}
    \]
    and that the relabelling of the vertices agrees with the geometric
    action of the two maps. This proves the case $n=1$. We also
    observe that the partition $\mathscr{P}_1$ is obtained through the
    line segment $\widebar{\ell}$. For the inductive step, suppose
    that, for a certain $n\geq 1$,
    $\mathscr{P}_n = \{\triangle_\omega \,:\, |\omega|=n\}$ and that
    the labelling of the vertices of each $\triangle_\omega$ is
    induced by $\phi_\omega$ as in (i). Consider a triangle
    $\triangle_\omega=\phi_{\omega}(\trianglecl)\in \mathscr{P}_n$,
    and note that
    \[
        \triangle_\omega=\phi_\omega(\trianglecl) = \phi_\omega(
        \phi_0(\trianglecl)\cup \phi_1(\trianglecl))=\phi_{\omega
          0}(\trianglecl)\cup \phi_{\omega 1}(\trianglecl).
    \]
    This partition is obtained through $\widebar{\phi_\omega(\ell)}$
    and we now prove that it agrees with the definition of
    $\mathscr{P}_{n+1}$. Indeed, from Lemma~\ref{lemma:preserve_Fsum},
    $\phi_\omega$ preserves the mediant, so that
    $\widebar{\phi_\omega(\ell)}$ joins $\phi_\omega(v_1)$ with
    $\phi_\omega(v_0\oplus
    v_2)=\phi_\omega(v_0)\oplus\phi_\omega(v_2)$. Thus
    $\phi_{\omega 0}(\trianglecl),\,\phi_{\omega 1}(\trianglecl)\in
    \mathscr{P}_{n+1}$. This proves
    $\{\triangle_\omega \,:\, |\omega|=n+1\}\subseteq
    \mathscr{P}_{n+1}$, and the two sets are in fact the same since
    they have the same cardinality. It remains to show that the
    labelling of the vertices of $\phi_{\omega 0}(\trianglecl)$ and
    $\phi_{\omega 1}(\trianglecl)$ according to the definition of
    $\mathscr{P}_{n+1}$ is induced by $\phi_{\omega 0}$ and
    $\phi_{\omega 1}$. This immediately follows by computing the
    images of the vertices of $\trianglecl$ under $\phi_{\omega 0}$
    and $\phi_{\omega 1}$.%
    \\[0.2cm] (ii) From (i) we have that
    $\set{\widebar{\phi_\omega(\ell)} \,:\, |\omega|= n}$ contains the
    line segments needed to pass from $\mathscr{P}_n$ to
    $\mathscr{P}_{n+1}$.
\end{proof}

In light of (ii) of the previous Lemma we have, for $n\geq -1$,
\[
    \mathcal{S}_{n+1} =%
    \bigcup_{\mathfrak{S}\in \mathscr{S}_{n+1}} (\mathfrak{S}\cap
    \mathcal{S}_n)^\oplus =%
    \bigcup_{\mathfrak{S}\in \mathscr{S}_{0}} (\mathfrak{S}\cap
    \mathcal{S}_n)^\oplus \cup \bigcup_{|\omega|\leq n}
    (\widebar{\phi_{\omega}(\ell)} \cap \mathcal{S}_n)^\oplus.
\]
By Lemma~\ref{lemma:char_Bn} and by the characterisation of the levels
of the Farey tree in terms of Stern-Brocot sets, it is easy to verify
that
\[
    \bigcup_{\mathfrak{S}\in \mathscr{S}_{0}} (\mathfrak{S}\cap
    \mathcal{S}_n)^\oplus = \mathcal{B}_{\leq n+1}.
\]
Hence
\begin{equation}\label{eq:S_n+1}
    \mathcal{S}_{n+1} = \mathcal{B}_{\leq n+1} \cup
    \bigcup_{|\omega|\leq n} (\widebar{\phi_{\omega}(\ell)} \cap
    \mathcal{S}_n)^\oplus,
\end{equation}
and this leads us towards studying the interior points of the
counterimages tree in order to prove that the two trees coincide level
by level.

\begin{proposition}\label{prop:char_In}
    For $n\geq 1$, the following properties hold.
    \begin{enumerate}[label={\upshape(\roman*)},wide = 0pt,leftmargin=*]
      \item
        $\ell \cap \mathcal{I}_n =
        \phi_1\left(\{x=1\}\cap\mathcal{B}_{n-1}\right)$ and
        $\#\left(\ell \cap\mathcal{I}_n \right) = 2^{n-1}$.
      \item Let $\mathfrak{p}$ and $\mathfrak{q}$ be the two endpoints
        of $\ell$, then
        \[
            \{\mathfrak{p}, \mathfrak{q}\}\cup \left(\ell\cap
              \mathcal{I}_{\leq n}\right) = \left(\{\mathfrak{p},
              \mathfrak{q}\}\cup \left(\ell\cap \mathcal{I}_{\leq
                  n-1}\right)\right)^\oplus,
        \]
        that is, the interior points up to level $n$ on $\ell$ are
        obtained from those up to level $n-1$ by inserting mediants of
        neighbouring points, where here the endpoints of $\ell$ are
        also included.
    \end{enumerate}
    Let $\omega\in \{0,1\}^*$ be a binary word of finite length
    $|\omega|\leq n-1$. Then the following properties hold.
    \begin{enumerate}[label={\upshape(\roman*)},wide = 0pt,leftmargin=*]
        \setcounter{enumi}{2}
      \item
        $\phi_\omega(\ell)\cap \mathcal{I}_n = \phi_{\omega}\left(\ell
          \cap \mathcal{I}_{n-|\omega|}\right)$ and
        $\#\left(\phi_\omega(\ell)\cap \mathcal{I}_n\right) =
        2^{n-|\omega|-1}$.
      \item Let
        $\mathfrak{p}_\omega\coloneqq \phi_{\omega}(\mathfrak{p})$ and
        $\mathfrak{q}_\omega\coloneqq \phi_{\omega}(\mathfrak{q})$ be
        the two endpoints of $\phi_\omega(\ell)$. Then
        \[
            \{\mathfrak{p}_\omega, \mathfrak{q}_\omega\}\cup
            \left(\phi_\omega(\ell) \cap \bigcup_{k=|\omega|+1}^{n}
              \mathcal{I}_k\right) = \left(\{\mathfrak{p}_\omega,
              \mathfrak{q}_\omega\}\cup \left(\phi_\omega(\ell) \cap
                \bigcup_{k=|\omega|+1}^{n-1} \mathcal{I}_k
              \right)\right)^\oplus.
        \]
    \end{enumerate}
\end{proposition}
\begin{proof}
    (i) The function $\phi_1$ bijectively maps the open vertical side
    $\{(1,y)\,:\, 0<y<1\}$ of $\trianglecl$ onto $\ell$. Moreover,
    since $n\geq 1$, by \ref{rule4} $\phi_1$ sends points of
    $\mathcal{B}_{n-1}$ to points in $\mathcal{I}_n$. The cardinality
    computation immediately follows from the first part and
    Lemma~\ref{lemma:char_Bn}.\\[0.2cm]
    (ii) From Lemma~\ref{lemma:char_Bn} we know that, for $n\geq 1$,
    the set $\{x=1\}\cap \mathcal{B}_{\leq n-1}$ corresponds to the
    Stern-Brocot set of level $n-1$. More precisely,
    $\{x=1\}\cap \mathcal{B}_{\leq n-1} = \set{(1,y) \,:\, y\in
      \mathcal{F}_{n-1}}$. Thus on the vertical side $\{x=1\}$ of
    $\trianglecl$, the points up to level $n-1$ are obtained from
    those up to level $n-2$ by taking mediants between neighbouring
    points. In other words
    \[
        \{x=1\}\cap \mathcal{B}_{\leq n-1} = \left(\{x=1\}\cap
          \mathcal{B}_{\leq n-2}\right)^\oplus.
    \]
    We now apply $\phi_1$ to both sides of the previous equality,
    getting
    \[
        \phi_1\left(\{x=1\}\cap \mathcal{B}_{\leq n-1}\right) =%
        \bigcup_{k=-1}^{n-1}\phi_1\left(\{x=1\}\cap
          \mathcal{B}_k\right) \overset{\text{(i)}}{=}%
        \{\mathfrak{p}, \mathfrak{q}\} \cup \left(\ell\cap
          \mathcal{I}_{\leq n}\right)
    \]
    and, using~\eqref{eq:f_Roplus},
    \[
        \phi_1\left(\left(\{x=1\}\cap \mathcal{B}_{\leq
              n-2}\right)^\oplus\right) =%
        \left(\phi_1\left(\{x=1\}\cap \mathcal{B}_{\leq
              n-2}\right)\right)^\oplus = \left(\{\mathfrak{p},
          \mathfrak{q}\}\cup \left(\ell\cap \mathcal{I}_{\leq
              n-1}\right)\right)^\oplus.
    \]
    \\[0.2cm] (iii) In case $|\omega|=0$ the first part is trivial and
    the second one has been proved in (i). Thus we can consider
    $1\leq |\omega|\leq n-1$. The function $\phi_\omega$ bijectively
    maps the open segment $\ell$ onto $\phi_\omega(\ell)$. Moreover,
    by applying \ref{rule4} $|\omega|$ times, $\phi_\omega$ maps
    points of $\mathcal{I}_{n-|\omega|}$ to points in
    $\mathcal{I}_n$. For the second part, (i) implies that
    $\#\left(\phi_\omega(\ell)\cap
      \mathcal{I}_n\right)=\#\left(\ell\cap
      \mathcal{I}_{n-|\omega|}\right)= 2^{n-|\omega|-1}$. %
    \\[0.2cm]
    (iv) From (ii) we know that for all $n\geq 1$ the interior points
    up to level $n$ on $\ell$ are obtained from those up to level
    $n-1$ by inserting mediants of neighbouring points, also
    considering the endpoints of $\ell$. Since $\phi_\omega$ preserves
    mediants, we can conclude applying $\phi_\omega$ to both sides of
    the equality in (ii).
\end{proof}

The above proposition characterises the location in $\trianglecl$ of
the interior points of our tree. In particular, it holds that
\[
    \mathcal{I}_n = \bigcup_{|\omega|\leq n-1}
    \left(\phi_\omega(\ell)\cap\mathcal{I}_n\right),
\]
that is, the interior points of level $n$ are located along the
backward images of $\ell$ under compositions of $\phi_0$ and $\phi_1$
of length $\leq n-1$. To prove this, note that the inclusion
``$\supseteq$'' is trivial and that the two sets have the same
cardinality. Indeed, using Proposition~\ref{prop:char_In}-(iii),
\[
    \#\left(\bigcup_{|\omega|\leq n-1}
      \left(\phi_\omega(\ell)\cap\mathcal{I}_n\right)\right) =%
    \sum_{s=0}^{n-1}\sum_{|\omega|=s}
    \#\left(\phi_\omega(\ell)\cap\mathcal{I}_n\right) =%
    \sum_{s=0}^{n-1}\sum_{|\omega|=s} 2^{n-s-1}=n2^{n-1}.
\]
As last step, we now write the set of the interior points up to level
$n+1$ in a convenient way. For $n\geq 1$,
\begin{align}\label{eq:I_n+1}
    \begin{split}
        \mathcal{I}_{\leq n+1} &= %
        \bigcup_{|\omega|\leq n} \left(\phi_\omega(\ell)\cap
          \mathcal{I}_{\leq n+1}\right) =%
        \bigcup_{|\omega|\leq n} \left(\{\mathfrak{p}_\omega,
          \mathfrak{q}_\omega\} \cup \left(\phi_\omega(\ell)\cap
            \mathcal{I}_{\leq n}\right)\right)^\oplus=\\
        &= \bigcup_{|\omega|\leq n}
        \left(\widebar{\phi_\omega(\ell)}\cap
          \left(\{\mathfrak{p}_\omega, \mathfrak{q}_\omega\} \cup
            \mathcal{I}_{\leq n}\right)\right)^\oplus.
    \end{split}
\end{align}

\begin{theorem} \label{thm:tree-mediant}
    For all $n\geq 0$ we have
    $\mathcal{T}_n=\mathcal{S}_n\setminus \mathcal{S}_{n-1}$, that is
    the tree defined by counterimages and that defined by Farey sums
    coincide level by level.
\end{theorem}
\begin{proof}
    It suffices to show that, for all $n\geq 0$,
    $\mathcal{S}_{n}=\mathcal{T}_{\leq n}=\mathcal{B}_{\leq n}\cup
    \mathcal{I}_{\leq n}$. We argue by induction on $n\geq 0$. Since
    $\mathcal{S}_0=\mathcal{T}_{-1}\cup \mathcal{T}_0$, the base case
    is proved. Now suppose that, for a certain $n\geq 0$,
    $\mathcal{S}_n=\mathcal{B}_{\leq n}\cup \mathcal{I}_{\leq n}$. By
    applying the inductive hypothesis we have
    \begin{align*}
        \mathcal{S}_{n+1} &\overset{\eqref{eq:S_n+1}}{=}%
        \mathcal{B}_{\leq n+1} \cup \bigcup_{|\omega|\leq n}
        (\widebar{\phi_{\omega}(\ell)} \cap \mathcal{S}_n)^\oplus =%
        \mathcal{B}_{\leq n+1} \cup \bigcup_{|\omega|\leq n}
        \left(\widebar{\phi_{\omega}(\ell)} \cap
          \left(\mathcal{B}_{\leq n}
            \cup \mathcal{I}_{\leq n} \right)\right)^\oplus =\\
        &= \mathcal{B}_{\leq n+1} \cup \bigcup_{|\omega|\leq n}
        \left(\widebar{\phi_{\omega}(\ell)} \cap
          \left(\{\mathfrak{p}_\omega,\mathfrak{q}_\omega\}\cup
            \mathcal{I}_{\leq n} \right)\right)^\oplus,
    \end{align*}
    where the last equality holds since, along
    $\widebar{\phi_{\omega}(\ell)}$, the points are all in
    $\mathcal{I}_{\leq n}$, with the possible exception of the
    endpoints of $\phi_{\omega}(\ell)$, namely $\mathfrak{p}_\omega$
    and $\mathfrak{q}_\omega$. Equation~\eqref{eq:I_n+1} allows us to
    conclude that
    $\mathcal{S}_{n+1} = \mathcal{B}_{\leq n+1}\cup \mathcal{I}_{\leq
      n+1}$, and the inductive step is proved.
\end{proof}

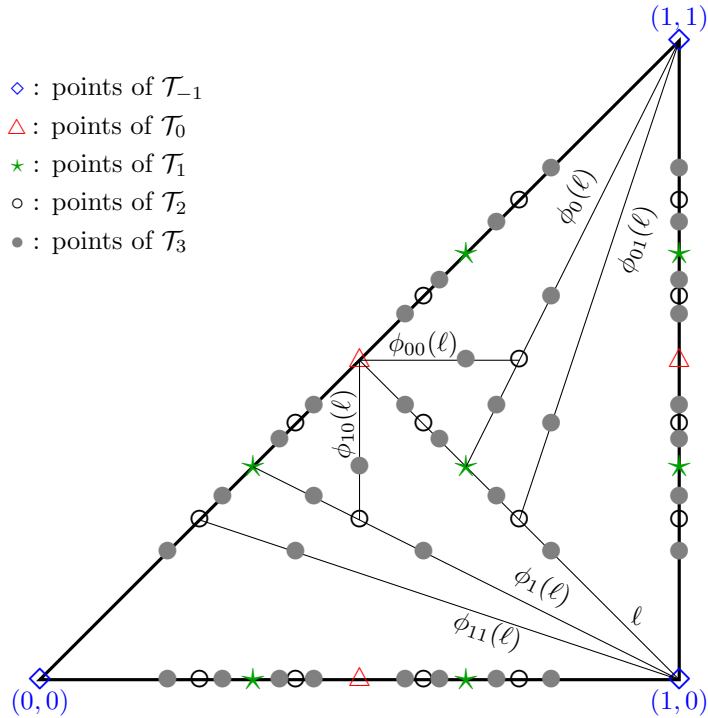
\begin{figure}[h]
    \begin{tikzpicture}[scale=8.5]
    \def\s{2}
    \def\o{0.075}
    \def\l{0.06}
    \draw[very thick] (0,0) -- (1,0) -- (1,1) -- cycle;

    \node[right] at (-\l,1-\o) {\textcolor{blue}{$\diamond$} : points of $\mathcal{T}_{-1}$};
    \node[right] at (-\l,1-\l-\o) {\textcolor{red}{\hspace{-0.05cm}\footnotesize$\triangle$\hspace{-0.05cm}} : points of $\mathcal{T}_{0}$};
    \node[right] at (-\l,1-2*\l-\o) {\textcolor{dgreen}{$\star$} : points of $\mathcal{T}_{1}$};
    \node[right] at (-\l,1-3*\l-\o) {$\circ$ : points of $\mathcal{T}_{2}$};
    \node[right] at (-\l,1-4*\l-\o) {\textcolor{gray}{$\bullet$} : points of $\mathcal{T}_3$};

    \node[blue] at (0,0) {\huge $\diamond$};
    \node[blue] at (1,0) {\huge $\diamond$};
    \node[blue] at (1,1) {\huge $\diamond$};
    \node[blue,below] at (0,0) {$(0,0)$};
    \node[blue,below] at (1,0) {$(1,0)$};
    \node[blue,above] at (1,1) {$(1,1)$};

    \node[red] at (1/2,0)   {$\triangle$};
    \node[red] at (1,1/2)   {$\triangle$};
    \node[red] at (1/2,1/2) {$\triangle$};

    \draw (1,0) -- (1/2,1/2);
    \node[dgreen] at (1/3,0)   {\huge $\star$};
    \node[dgreen] at (2/3,0)   {\huge $\star$};
    \node[dgreen] at (1,1/3)   {\huge $\star$};
    \node[dgreen] at (1,2/3)   {\huge $\star$};
    \node[dgreen] at (1/3,1/3) {\huge $\star$};
    \node[dgreen] at (2/3,2/3) {\huge $\star$};
    \node[dgreen] at (2/3,1/3) {\huge $\star$};

    \draw (1,1) -- (2/3,1/3);
    \draw (1,0) -- (1/3,1/3);
    \node at (1/4,0) {\huge $\circ$};
    \node at (2/5,0) {\huge $\circ$};
    \node at (3/5,0) {\huge $\circ$};
    \node at (3/4,0) {\huge $\circ$};
    \node at (1,1/4) {\huge $\circ$};
    \node at (1,2/5) {\huge $\circ$};
    \node at (1,3/5) {\huge $\circ$};
    \node at (1,3/4) {\huge $\circ$};
    \node at (1/4,1/4) {\huge $\circ$};
    \node at (2/5,2/5) {\huge $\circ$};
    \node at (3/5,3/5) {\huge $\circ$};
    \node at (3/4,3/4) {\huge $\circ$};
    \node at (3/5,2/5) {\huge $\circ$};
    \node at (3/4,1/4) {\huge $\circ$};
    \node at (3/4,2/4) {\huge $\circ$};
    \node at (2/4,1/4) {\huge $\circ$};

    \draw (1,0) -- (1/4,1/4);
    \draw (1/2,1/2) -- (1/2,1/4);
    \draw (1/2,1/2) -- (3/4,1/2);
    \draw (1,1) -- (3/4,1/4);
    \node[gray] at (1/5,0) {\huge $\bullet$};
    \node[gray] at (2/7,0) {\huge $\bullet$};
    \node[gray] at (3/8,0) {\huge $\bullet$};
    \node[gray] at (3/7,0) {\huge $\bullet$};
    \node[gray] at (4/7,0) {\huge $\bullet$};
    \node[gray] at (5/8,0) {\huge $\bullet$};
    \node[gray] at (5/7,0) {\huge $\bullet$};
    \node[gray] at (4/5,0) {\huge $\bullet$};
    \node[gray] at (1,1/5) {\huge $\bullet$};
    \node[gray] at (1,2/7) {\huge $\bullet$};
    \node[gray] at (1,3/8) {\huge $\bullet$};
    \node[gray] at (1,3/7) {\huge $\bullet$};
    \node[gray][gray] at (1,4/7) {\huge $\bullet$};
    \node[gray] at (1,5/8) {\huge $\bullet$};
    \node[gray] at (1,5/7) {\huge $\bullet$};
    \node[gray] at (1,4/5) {\huge $\bullet$};
    \node[gray] at (1/5,1/5) {\huge $\bullet$};
    \node[gray] at (2/7,2/7) {\huge $\bullet$};
    \node[gray] at (3/8,3/8) {\huge $\bullet$};
    \node[gray] at (3/7,3/7) {\huge $\bullet$};
    \node[gray] at (4/7,4/7) {\huge $\bullet$};
    \node[gray] at (5/8,5/8) {\huge $\bullet$};
    \node[gray] at (5/7,5/7) {\huge $\bullet$};
    \node[gray] at (4/5,4/5) {\huge $\bullet$};
    \node[gray] at (4/7,3/7) {\huge $\bullet$};
    \node[gray] at (5/8,3/8) {\huge $\bullet$};
    \node[gray] at (5/7,2/7) {\huge $\bullet$};
    \node[gray] at (4/5,1/5) {\huge $\bullet$};
    \node[gray] at (3/7,2/7) {\huge $\bullet$};
    \node[gray] at (3/5,1/5) {\huge $\bullet$};
    \node[gray] at (5/7,3/7) {\huge $\bullet$};
    \node[gray] at (4/5,3/5) {\huge $\bullet$};
    \node[gray] at (2/3,1/2) {\huge $\bullet$};
    \node[gray] at (1/2,1/3) {\huge $\bullet$};
    \node[gray] at (4/5,2/5) {\huge $\bullet$};
    \node[gray] at (2/5,1/5) {\huge $\bullet$};

    \node[right] at (0.91,0.1) {$\ell$};
    \node[above,rotate=atan(2)] at (0.87,2*0.87-1) {$\phi_0(\ell)$};
    \node[above,rotate=-atan(1/2)] at (0.77,-1/2*0.77+1/2) {$\phi_1(\ell)$};
    \node at (1/2+0.1,1/2+0.025) {$\phi_{00}(\ell)$};
    \node[rotate=90] at (1/2-0.025,1/2-0.1) {$\phi_{10}(\ell)$};
    \node[rotate=atan(-1/3)] at (0.7,-1/4*0.7+1/4) {$\phi_{11}(\ell)$};
    \node[below,rotate=atan(3)] at (0.9,3*0.9-2) {$\phi_{01}(\ell)$};
\end{tikzpicture}
    \caption{The first four levels of the tree generated through the
      local inverses of the map $\tilde{S}$ represented as set of
      points of $\trianglecl$.}\label{fig:triangle_levels}
\end{figure}

\begin{remark}
    At this point it is natural to ask, given the map
    $S$ and the tree defined here, whether or not a version of the
    Minkowski question mark function could be defined in this
    setting. We recall, briefly, that the original Minkowski question
    mark function was introduced as another way of demonstrating the
    Lagrange property of continued fractions, in that it maps every
    rational number to the subset of dyadic rationals (that is, those
    having denominators containing only powers of 2) and every
    quadratic irrational to the remaining rational numbers (see
    \cite{min, mink?}, and for other 1-dimensional analogues,
    \cite{SMJJM, SM1}). These functions are now known as {\em slippery
      Devil's staircases} for the fact that they are strictly
    increasing but nevertheless singular with respect to the Lebesgue
    measure.

    This natural question has been studied in \cite{mink} for the map $S$ and many possible generalisations. A higher dimensional version of the Minkowski function for a different map has been introduced in  \cite{Panti}.
    \end{remark}


\appendix

\section{Some results on the local inverses of $V$}
\label{sec:matrices}

In this appendix we prove some properties of the local inverses of the
map $V$, needed for the argument of Section~\ref{subsec:proof_pde}. We
recall that $V$ is the induced map of $S$ on the set
$A=\{(x,y)\in \Gamma_0 \,:\, S(x,y)\in \Gamma_0\}$, and that each local
inverse of $V$ is a linear fractional map, as in~\eqref{form-psi}. In
general, a linear fractional map $\psi$ of the form
\[
    \psi(x,y) =  \left( \frac{r_1 +s_1 x + t_1 y}{r+sx+ty},\ \frac{r_2
        +s_2 x + t_2 y}{r+sx+ty}\right),
\]
where the cofficients are non-negative integers, can be expressed in
projective coordinates by the $3\times 3$ matrix
\[
    M_{\psi} \coloneqq
    \begin{pmatrix}
       r & s & t\\
       r_1 & s_1 & t_1 \\
       r_2 & s_2 & t_2
    \end{pmatrix}
\]
by associating a point $(\frac{x}{z},\frac{y}{z})\in \R^2$ to a vector $v=( z,\, x,\, y)^t$, so that $\psi\left( \frac{x}{z},\frac{y}{z} \right)$ is associated to the vector $M_\psi v$. For instance, the two inverse maps $\phi_0$ and $\phi_1$ have matrices
\[
    M_0 \coloneqq M_{\phi_0} =
    \begin{pmatrix}
    1 & 0 & 1 \\
    1 & 0 & 0 \\
    0 & 1 & 0
    \end{pmatrix}
    \qquad\text{and}\qquad M_1 \coloneqq M_{\phi_1} =
    \begin{pmatrix}
    1 & 0 & 1 \\
    0 & 1 & 0 \\
    0 & 0 & 1
    \end{pmatrix}.
\]
Note that the composition of linear fractional maps translates into
the left multiplication of their matrices. As a consequence, since
both $M_0$ and $M_1$ have unit determinant, every product involving
these two matrices also has unit determinant. To every linear
fractional map as above, we associate the vectors
$v_1(\psi),\,v_2(\psi),\,v(\psi) \in \R^3$ corresponding to the rows
of the associated matrix $M_\psi$. In other words,
\[
    v_1(\psi) = \begin{pmatrix}
    r_1 \\
    s_1 \\
    t_1
\end{pmatrix},
    \qquad
    v_2(\psi) =
    \begin{pmatrix}
    r_2 \\
    s_2 \\
    t_2
\end{pmatrix},
\qquad
    v(\psi) =
    \begin{pmatrix}
    r \\
    s \\
    t
\end{pmatrix}.
\]

In what follows, we use the notation $\norm{\cdot}$ for the Euclidean
norm and $\norm{\cdot}_1$ for the $1$-norm on $\R^3$. The two norms
are equivalent and, in particular, for all $v\in \R^3$ holds
\begin{equation}
    \label{norme-equiv}
    \| v\| \le \|v\|_1 \le \sqrt{3} \|v\|.
\end{equation}
Moreover, to each $v\in \R^3\setminus\{0\}$ with non-negative
components we associate the normalised vector $P_v$ given by
$P_v \coloneqq \frac{v}{\| v\|_1}$.
\begin{lemma}\label{lem-geo}
    For any $v,w\in \R^3\setminus\{0\}$ with non-negative components,
    $ \| v \times w\| \le \sqrt{3} \|v\|\|w\|\|P_v - P_w\|.  $
\end{lemma}
\begin{proof}
    Let $\theta_{v,w}\in \left[0,\frac{\pi}2\right]$ be the angle
    between the two vectors $v$ and $w$, so that
    $\|v\times w\| = \|v\|\|w\|\sin\theta_{v,w}$, and between the two
    vectors $P_v$ and $P_w$. Let
    $\lambda_{w,u}\coloneqq\| P_w -(P_w\cdot P_v) P_v\|$, the modulus
    of the component of $P_w$ orthogonal to $P_v$ (see
    Figure~\ref{fig:lemma_vxw}). Then by simple geometric
    considerations we have
    \[
        \sin\theta_{v,w} = \frac{\lambda_{w,v}}{\|P_w\|} =
        \frac{\lambda_{w,v}}{\|w\|}\|w\|_1
        \overset{\eqref{norme-equiv}}{\leq} \sqrt{3}\lambda_{w,v} \leq
        \sqrt{3}\|P_v - P_w\|,
    \]
\end{proof}

\begin{figure}[h!]
    \begin{tikzpicture}[scale=2.5]
    \def\al{30}
    \def\l{1.5}

    \draw[thick,-latex'] (0,0) -- (2.5,0) node[below]{\footnotesize $v$};
    \draw[thick,-latex'] (0,0) -- ({2*cos(\al)},{2*sin(\al)}) node[above]{\footnotesize $w$};

    \draw[thick,-latex',DBlu] (0,0) -- (\l,0) node[below]
    {\footnotesize $P_v$}; \draw[thick,-latex',DBlu] (0,0) --
    ({2*\l *cos(\al)/2},{2*\l *sin(\al)/2}) node[above left]
    {\footnotesize $P_w$};

    \draw[thin,densely dotted] ({2*\l *cos(\al)/2},{2*\l *sin(\al)/2}) --
    ({2*\l *cos(\al)/2},0);

    \node[below left] at ({2*\l *cos(\al)/2},{2*\l *sin(\al)/4})
    {\footnotesize $\lambda_{w,v}$};

    \draw[thick,latex'-] (\l,0)--({2*\l *cos(\al)/2},{2*\l *sin(\al)/2});

    \node[rotate=-(180-\al)/2,above] at
    ({(\l +2*\l *cos(\al)/2)/2},{\l *sin(\al)/2}) {\footnotesize $P_v-P_w$};

    \centerarcfill[fill=red,opacity=0.2](0,0)(0:\al:0.4);
    \node[red] at (0.55,0.13) {\footnotesize $\theta_{v,w}$};
\end{tikzpicture}
    \caption{Graphical representation of vectors $P_v$ and $P_w$, along
      with the quantities involved in the proof of
      Lemma~\ref{lem-geo}.}\label{fig:lemma_vxw}
\end{figure}
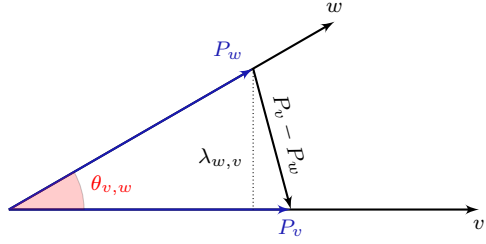

\begin{lemma}\label{lem:matrix_ineq}
    Let $\Phi$ be an arbitrary composition of the maps $\phi_0$ and
    $\phi_1$. Then the matrix $M_\Phi$ satisfies
    $\|v(\Phi)\|_1\geq \|v_1(\Phi)\|_1$ and
    $\|v(\Phi)\|_1\geq \|v_2(\Phi)\|_1$.
\end{lemma}
\begin{proof}
    We argue by induction on the length $l\geq 1$ of $\Phi$ as a
    composition of maps. If $l=1$, then $\Phi$ is either $\phi_0$ or
    $\phi_1$, and in both cases the thesis is true. For the inductive
    step, let $l\geq 1$ and suppose that the thesis is true for a
    certain $\Phi$ of length $l$. Let
    \[
        M_{\Phi} = \begin{pmatrix}
           r & s & t\\
       r_1 & s_1 & t_1 \\
       r_2 & s_2 & t_2
        \end{pmatrix}
    \]
    be the matrix of $\Phi$. We have
    \[
        M_{\phi_0\circ\Phi} = M_0M_\Phi =
        \begin{pmatrix}
            r+r_2 & s+s_2 & t+t_2\\
            r & s & t \\
            r_1 & s_1 & t_1
        \end{pmatrix}
        \quad\text{and}\quad M_{\phi_1\circ\Phi} = M_1M_\Phi =
        \begin{pmatrix}
            r+r_2 & s+s_2 & t+t_2 \\
            r_1 & s_1 & t_1 \\
            r_2 & s_2 & t_2
        \end{pmatrix}.
    \]
    For the first matrix it holds that
    \[
        \|v(\phi_0\circ\Phi)\|_1=r+s+t+r_2+s_2+t_2\geq r+s+t =
        \|v_1(\phi_0\circ \Phi)\|_1
    \]
    since $r_2,\,s_2,\,t_2\geq 0$, and that
    \[
        \|v(\phi_0\circ\Phi)\|_1\geq r+s+t\geq
        r_1+s_1+t_1=\|v_2(\phi_0\circ \Phi)\|_1
    \]
    by the inductive assumption. Analogous estimates hold for
    $M_{\phi_1\circ \Phi}$.
\end{proof}

\begin{lemma}\label{lem-matrici}
    Let $\Phi$ be an arbitrary composition of the maps $\phi_0$ and
    $\phi_1$.
    \begin{enumerate}[label={\upshape(\roman*)},wide =
        0pt,leftmargin=*]
      \item If $D\Phi$ denotes the Jacobian matrix of $\Phi$, then
        \[ \begin{aligned}
            \max \left\{ \sup_A\left(\left|(D\Phi)_{11} \right| +
                \left|(D\Phi)_{21}\right|\right),\, \sup_A\left(
                \left|(D\Phi)_{12} \right| +
                \left| (D\Phi)_{22}\right|\right)\right\}\leq  \\
            &\hspace{-5cm}\le 27\sqrt{3} \left(\norm{P_{v(\Phi)}-P_{v_1(\Phi)}}+
              \norm{P_{v(\Phi)}-P_{v_2(\Phi)}}\right).\nonumber
        \end{aligned} \]
      \item For $k=0,1$, let $D_k$ be the Jacobian matrix of
        $\phi_k\circ \Phi$, then
        \[ \begin{aligned}
            \max \left\{ \sup_A\left(\left|(D_k)_{11} \right| +
                \left|(D_k)_{21}\right|\right),\, \sup_A\left(
                \left|(D_k)_{12} \right| + \left|
                  (D_k)_{22}\right|\right)\right\}\leq \\
            &\hspace{-5cm}\le 27\sqrt{3}
            \left(\norm{P_{v(\Phi)+v_2(\Phi)}-P_{v_1(\Phi)}}+
              \norm{P_{v(\Phi)+v_2(\Phi)}-P_{v_2(\Phi)}}\right).\nonumber
        \end{aligned} \]
        In particular the worst case is realised for $k=0$.
      \item We have that
        \[
            \norm{P_{v(\Phi)+v_2(\Phi)}-P_{v_1(\Phi)}}+
            \norm{P_{v(\Phi)+v_2(\Phi)}-P_{v_2(\Phi)}} \leq
            \norm{P_{v(\Phi)}-P_{v_1(\Phi)}}+
            \norm{P_{v(\Phi)}-P_{v_2(\Phi)}}.
        \]
    \end{enumerate}
\end{lemma}
\begin{proof}
    (i) The map $\Phi$ is a composition of $\phi_0$ and $\phi_1$, thus
    of the form~\eqref{form-dpsi}. Since we are interested in the local inverses of $V$, we look at the supremum of the Jacobian matrix of $\Phi$ on $A$. Hence we have
    \begin{align*}
        \sup_A\left( \left| (D\Phi)_{11} \right| + \left|
            (D\Phi)_{21}\right| \right) &\leq 9 \cdot
        \frac{\abs{r_1s-rs_1} + \abs{s_1t-st_1} +
          \abs{r_2s-rs_2} + \abs{s_2t-st_2}}{(r+s+t)^2}\leq \\
        &\leq 9\cdot \frac{\norm{v(\Phi)\times v_1(\Phi)}_1 +
          \norm{v(\Phi)\times v_2(\Phi)}_1}{\norm{v(\Phi)}^2}\leq\\
        &\overset{\eqref{norme-equiv}}{\leq} 9\sqrt{3}\cdot
        \frac{\norm{v(\Phi)\times v_1(\Phi)} + \norm{v(\Phi)\times
            v_2(\Phi)}}{\norm{v(\Phi)}^2}\leq\\
        &\overset{\textrm{Lem.~\ref{lem-geo}}}{\leq} 27\cdot
        \frac{\norm{v_1(\Phi)}\norm{P_{v(\Phi)}-P_{v_1(\Phi)}} +
          \norm{v_2(\Phi)}\norm{P_{v(\Phi)}-P_{v_2(\Phi)}}}{\norm{v(\Phi)}}\leq\\
        &\overset{\eqref{norme-equiv}}{\leq}27\sqrt{3}\cdot
        \frac{\norm{v_1(\Phi)}_1\norm{P_{v(\Phi)}-P_{v_1(\Phi)}} +
          \norm{v_2(\Phi)}_1\norm{P_{v(\Phi)}-P_{v_2(\Phi)}}}{\norm{v(\Phi)}_1}.
    \end{align*}
    From Lemma~\ref{lem:matrix_ineq} we have
    $\norm{v_1(\Phi)}_1\leq \norm{v(\Phi)}_1$ and
    $\norm{v_2(\Phi)}_1\leq \norm{v(\Phi)}_1$, so that
    \[
      \sup_A\, \left( \left| (D\Phi)_{11} \right| +
        \left| (D\Phi)_{21}\right| \right)\leq
      27\sqrt{3}\cdot
      \left(\norm{P_{v(\Phi)}-P_{v_1(\Phi)}}+ \norm{P_{v(\Phi)}-P_{v_2(\Phi)}}\right)
    \]
    The same estimate holds for
    $\sup_A\left( \left| (D\Phi)_{12} \right| + \left|
        (D\Phi)_{22}\right| \right)$ and thus (i) is proved.\\[0.2cm]
    (ii) The matrices associated to the maps $\phi_k\circ \Phi$ for
    $k=0,1$ are
    \[
        M_0M_\Phi =
        \begin{pmatrix}
            v(\Phi)+v_2(\Phi)\\v(\Phi)\\v_1(\Phi)
        \end{pmatrix}
        \quad\text{and}\quad%
        M_1M_\psi =
        \begin{pmatrix}
            v(\Phi)+v_2(\Phi)\\v_1(\Phi)\\v_2(\Phi)
        \end{pmatrix}.
    \]
    Applying (i) to the map $\phi_k\circ \Phi$ we have
    \begin{align*}
        \max \left\{\sup_A\left( \left| (D_0)_{11} \right| + \left|
              (D_0)_{21}\right| \right),\, \sup_A \left(\left|
              (D_0)_{12}
            \right| + \left| (D_0)_{22}\right| \right)\right\} \leq\\
        &\hspace{-5cm}\leq 27\sqrt{3}
        \left(\norm{P_{v(\Phi)+v_2(\Phi)}-P_{v(\Phi)}}
          +\norm{P_{v(\Phi)+v_2(\Phi)}-P_{v_1(\Phi)}}\right)
    \end{align*}
    and
    \begin{align*}
        \max \left\{ \sup_A\left( \left| (D_1)_{11} \right| + \left|
              (D_1)_{21}\right| \right),\,\sup_A \left( \left|
              (D_1)_{12}
            \right| + \left| (D_1)_{22}\right| \right) \right\} \leq\\
        &\hspace{-5cm}\leq 27\sqrt{3}
        \left(\norm{P_{v(\Phi)+v_2(\Phi)}-P_{v_1(\Phi)}}+
          \norm{P_{v(\Phi)+v_2(\Phi)}-P_{v_2(\Phi)}}\right).
    \end{align*}
    To finish the proof it suffices to show that
    \begin{equation}\label{eq:dist_v_v2}
        \norm{P_{v(\Phi)+v_2(\Phi)}-P_{v(\Phi)}}\leq
        \norm{P_{v(\Phi)+v_2(\Phi)}-P_{v_2(\Phi)}}.
    \end{equation}
    To this end, note that $P_{v(\Phi)+v_2(\Phi)}$ is a convex
    combination of $P_{v(\Phi)}$ and $P_{v_2(\Phi)}$, in particular
    \[
        P_{v(\Phi)+v_2(\Phi)} =
        \frac{\norm{v(\Phi)}_1}{\norm{v(\Phi)}_1 + \norm{v_2(\Phi)}_1}
        P_{v(\Phi)} + \frac{\norm{v_2(\Phi)}_1}{\norm{v(\Phi)}_1 +
          \norm{v_2(\Phi)}_1}P_{v_2(\Phi)},
    \]
    and since from Lemma~\ref{lem:matrix_ineq} we have
    $\norm{v_2(\Phi)}_1 \leq \norm{v(\Phi)}_1$, \eqref{eq:dist_v_v2}
    easily follows.\\[0.2cm]
    (iii) The three points $P_{v(\Phi)}$, $P_{v_1(\Phi)}$ and
    $P_{v_2(\Phi)}$ belong to the standard 2-symplex in $\R^3$ and
    define a triangle $\triangle_\Phi$ since they are linearly
    independent. Furthermore, (ii) implies that
    $P_{v(\Phi)+v_2(\Phi)} = \lambda P_{v(\Phi)} + (1-\lambda)
    P_{v_2(\Phi)}$ for some $\lambda\in \left(\frac 12,1\right)$. Also
    $P_{v(\Phi)+v_2(\Phi)}$, $P_{v_1(\Phi)}$ and $P_{v_2(\Phi)}$
    define a triangle $\triangle_\Phi'$, which is a subtriangle of
    $\triangle_\Phi$. In particular, $\triangle_\Phi$ and
    $\triangle_\Phi'$ have a common side and the non-common vertex
    $P_{v(\Phi)+v_2(\Phi)}$ belongs to the side of $\triangle_\Phi$
    with vertices $P_{v(\Phi)}$ and $P_{v_2(\Phi)}$. The inequality to
    prove easily follows from this geometric interpretation, since
    perimeter of the subtriangle $\triangle_\Phi'$ is less than or
    equal to the perimeter of $\triangle_\Phi$. Besides this
    geometrical approach, an analytic estimate easily follows from the
    triangle inequality:
    \begin{align*}
        \norm{P_{v(\Phi)+v_2(\Phi)}-P_{v_1(\Phi)}} &+
        \norm{P_{v(\Phi)+v_2(\Phi)}-P_{v_2(\Phi)}}= \\ &=
        \norm{\lambda P_{v(\Phi)} + (1-\lambda)
          P_{v_2(\Phi)}-P_{v_1(\Phi)}} +
        \lambda \norm{P_{v(\Phi)}- P_{v_2(\Phi)}}= \\
        &=
        \norm{-(1-\lambda)(P_{v(\Phi)}-P_{v_2(\Phi)})+P_{v(\Phi)}-P_{v_1(\Phi)}}
        + \lambda \norm{P_{v(\Phi)}- P_{v_2(\Phi)}}\leq \\
        &\leq
        (1-\lambda)\norm{P_{v(\Phi)}-P_{v_2(\Phi)}}+\norm{P_{v(\Phi)}-P_{v_1(\Phi)}}
        + \lambda \norm{P_{v(\Phi)}- P_{v_2(\Phi)}}=\\
        &=\norm{P_{v(\Phi)}-P_{v_1(\Phi)}}+
        \norm{P_{v(\Phi)}-P_{v_2(\Phi)}}.
    \end{align*}
\end{proof}

\begin{lemma}\label{lem-decr}
    Let $\psi_{i_1,\,\dots,\,i_n} : A\rightarrow
    C_{i_1,\,\dots,\,i_n}$ be a local inverse of $V^n$. Then
    \[
    \norm{P_{v(\psi_{i_1,\,\dots,\,i_n})}-P_{v_1(\psi_{i_1,\,\dots,\,i_n})}}
    +
    \norm{P_{v(\psi_{i_1,\,\dots,\,i_n})}-P_{v_2(\psi_{i_1,\,\dots,\,i_n})}}
    \leq \tilde d(n)
    \]
    where
    \[
        \tilde d(n)\coloneqq \norm{P_{v(\phi_0^{n})}-P_{v_1(\phi_0^{n})}} +
        \norm{P_{v(\phi_0^{n})}-P_{v_2(\phi_0^{n})}}.
    \]
\end{lemma}
\begin{proof}
    As outlined in Section~\ref{subsec:proof_pde},
    $\psi_{i_1,\,\ldots,\,i_n} = \psi_{i_1}\circ\dots\circ\psi_{i_n}$
    and, for $h=1,\,\ldots,\,n$, $\psi_{i_h}=\phi_0\circ \Phi_{i_h}$,
    where $\Phi_{i_h}$ is empty or a composition of the maps $\phi_0$ and $\phi_1$ beginning with $\phi_0$. In Proposition~\ref{lem-matrici}~(iii) we
    proved that the estimation function introduced in (i) is
    decreasing with respect to the number of compositions of the maps
    $\phi_0$ or $\phi_1$. The inequality of this lemma follows, since
    $\phi_0^n$ contains the least possible number of compositions of
    the maps $\phi_0$ and $\phi_1$ compatible with the definition of
    the local inverses, and by Proposition \ref{lem-matrici}-(ii) realises the worst case.
\end{proof}

\begin{lemma}\label{lem-dn}
    Let $(\tilde d(n))_{n\geq 0}$ be the real sequence introduced in
    Lemma~\ref{lem-decr}. Then $\lim_{n\rightarrow+\infty} \tilde d(n) = 0$.
\end{lemma}
\begin{proof}
    Arguing by induction on $n\geq 0$, it is easy to prove that
    \[
    M_{\phi_0^n} = M_0^n =
    \begin{pmatrix}
    f_{n+4} & f_{n+2} & f_{n+3}\\
    f_{n+3} & f_{n+1} & f_{n+2}\\
    f_{n+2} & f_{n} & f_{n+1}
    \end{pmatrix},
    \]
    where $(f_n)_{n\geq 0}$ is recursively defined to be
    \[
    \begin{cases}
    f_0=0\\f_1=1\\f_2=0\\
    f_{n+3} = f_{n+2}+f_n & \text{for }n\geq 1
    \end{cases}.
    \]
    The sequence $(\nu_n)_{n\geq 0}$, $\nu_n\coloneqq f_{n+4}$, is
    also referred to as the Narayana's cows sequence. It is known that
    this sequence has a ratio limit, \emph{i.e.} there exists
    $\lim_{n\rightarrow \infty}\frac{\nu_{n+1}}{\nu_n} \eqqcolon
    \gamma <+\infty$ \cite{vernon}\footnote{More precisely, $\gamma$
      is the only real root of the characteristic equation
      $x^3-x^2-1=0$.}. Note that for $n\geq 4$ and for $r\geq 1$ we
    have
    $\frac{f_{n+r}}{f_n} = \prod_{j=0}^{r-1}
    \frac{f_{n+1+j}}{f_{n+j}}$, so that
    \begin{equation}\label{eq:narayana}
        \lim_{n\rightarrow +\infty} \frac{f_{n+r}}{f_n} = \gamma^r.
    \end{equation}
    For $n\geq 0$ we have
    \begin{align*}
        P_{v_1(\phi_0^{n})} &= \frac
        1{f_{n+1}+f_{n+2}+f_{n+3}}
        \begin{pmatrix}f_{n+3}\\f_{n+1}\\f_{n+2}\end{pmatrix} =
        \frac 1{f_{n+5}}
        \begin{pmatrix}f_{n+3}\\f_{n+1}\\f_{n+2}\end{pmatrix}\\
        P_{v_2(\phi_0^{n})} &= \frac
        1{f_{n}+f_{n+1}+f_{n+2}}
        \begin{pmatrix}f_{n+2}\\f_{n}\\f_{n+1}\end{pmatrix} =
        \frac 1{f_{n+4}}
        \begin{pmatrix}f_{n+2}\\f_{n}\\f_{n+1}\end{pmatrix},\\
        P_{v(\phi_0^{n})} &= \frac
        1{f_{n+2}+f_{n+3}+f_{n+4}}
        \begin{pmatrix}f_{n+4}\\f_{n+2}\\f_{n+3}\end{pmatrix} =
        \frac
        1{f_{n+6}}
        \begin{pmatrix}f_{n+4}\\f_{n+2}\\f_{n+3}\end{pmatrix}
    \end{align*}
    so that
    \begin{align*}
        \tilde d(n)&\leq \norm{P_{v(\phi_0^{n})}-P_{v_1(\phi_0^{n})}}_1 +
        \norm{P_{v(\phi_0^{n})}-P_{v_2(\phi_0^{n})}}_1 =\\
        &= \sum_{k=0}^2
        \left(\abs{\frac{f_{n+k+2}}{f_{n+6}}-\frac{f_{n+k+1}}{f_{n+5}}}
          +
          \abs{\frac{f_{n+k+2}}{f_{n+6}}-\frac{f_{n+k}}{f_{n+4}}}\right).
    \end{align*}
    Using \eqref{eq:narayana}, for each $k=0,\,1,\,2$ we have
    $\frac{f_{n+k+2}}{f_{n+6}}-\frac{f_{n+k+1}}{f_{n+5}}\rightarrow
    \gamma^{4-k}-\gamma^{4-k}=0$ and analogously
    $\frac{f_{n+k+2}}{f_{n+6}}-\frac{f_{n+k}}{f_{n+4}} \rightarrow
    0$ as $n\rightarrow+\infty$. This proves that
    $\lim_{n\rightarrow+\infty} \tilde d(n)= 0$.
\end{proof}

\begin{proof}[Proof of Proposition \ref{prop-estimate-dpsi}] It follows directly from Lemma~\ref{lem-matrici}-(i), \ref{lem-decr} and~\ref{lem-dn} with $d(n)=27 \sqrt{3}\, \tilde d(n)$.
\end{proof}

\section{The wandering rate of the set $A$}
\label{sec:varying-seq}

The set $A$ is defined in \eqref{the-set-A} and it is the triangle
with vertices $Q_1=(\frac 12, \frac 12)$, $Q_2= (\frac 23, \frac 13)$
and $Q_3=(1,1)$, with the sides $Q_1Q_2$ and $Q_2Q_3$ not included. We
consider the wandering rate $w_n(A)$ for $n\geq 1$, which is defined
to be
\[
    w_n(A) \coloneqq \sum_{k=0}^{n-1} \mu(A \cap \{\varphi >k\}),
\]
where $\varphi$ is the first-return time function in $A$. Extending
the function $\varphi$ to all $\trianglecl$ by
\[
    \varphi(x,y) \coloneqq \min \left\{ n\ge 1\, :\, S^n(x,y)\in
      A\right\}
\]
we obtain the \emph{hitting time} function of $A$, which is
well-defined and finite $\mu$-almost everywhere since the system
$(\trianglecl, \mu, S)$ is conservative and ergodic. We now recall
that, for $k\geq 1$,
\[
    \mu(A \cap \{ \varphi >k\}) = \mu(A^\complement  \cap \{ \varphi =k\}),
\]
where $A^\complement \coloneqq \trianglecl \setminus A$ \cite[Lemma 1]{zwei}. We
thus study the diverging sequence
$\sum_{k=1}^n \mu(A^\complement  \cap \{ \varphi =k\})$. The first step is to
study the structure of $A^\complement  \cap \{ \varphi =k\}$ for $k\geq 1$, the
set of points in $A^\complement $ which hit $A$ for the first time after exactly
$k$ iterations of the map $S$. This set can be expressed in terms of
the local inverse of $S$ as follows. Let
\[
    \Omega_k \coloneqq \set{\omega \in \{0,1\}^k \,:\, \pi_{\omega_i
        \omega_{i+1}}=1\ \forall i=0,\,\ldots,\,k-2,\ \omega_{k-1}=1},
    \quad\text{where }
    \Pi= (\pi_{ij})_{i,j=0,1} = \begin{pmatrix} 0 & 1\\
        1 & 1\end{pmatrix}.
\]
In this way, $\Omega_k$ is the set of binary words of length $k$, which all
end with a ``$1$'', and in which the string ``$00$'' never appears. Then
\[
    A^\complement  \cap \{\varphi=k\} =%
    \bigcup_{\omega\in \Omega_k} \phi_\omega(A) =%
    \bigcup_{\omega \in \Omega_k} \phi_{\omega_0} \circ
    \phi_{\omega_1} \circ \dots \circ \phi_{\omega_{k-2}} \circ \phi_1
    (A).
\]
Indeed, a point in
$\phi_{\omega_0} \circ \phi_{\omega_1} \circ \dots \circ
\phi_{\omega_{k-2}} \circ \phi_1 (A)$ has symbolic code given by
$\omega_0 \omega_1\dots\omega_{k-2} 1 00$ with the word ``$00$'' not
appearing in the first $k$ symbols. This is equivalent to saying that
such a point does not visit $A$ in the first $k-1$ iterations, hence
the point is in $A^\complement  \cap\{ \varphi =k\}$. The converse also obviously
holds. Note that, in case $k=1$, we have
$A^\complement \cap \{\varphi=1\}=\phi_1(A)$.

We first obtain an estimate from above for the wandering rate. In what
follows, we write $a_n \lesssim b_n$ if and only if $a_n=O(b_n)$.

\begin{proposition}\label{prop:w-above}
    The wandering rate $w_n(A)$ satisfies
    $w_n(A) \lesssim \log^2 n$.
\end{proposition}

\begin{proof}
    Using the properties of the map $S$ and its local inverses, one
    immediately  verifies that
    \[
        \bigcup_{k=1}^n \left(A^\complement  \cap \{ \varphi =k\}\right) \subseteq
        \bigcup_{k=0}^{n-1} \triangle_k,
    \]
    where $\{\triangle_k\}_{k\geq 0}$ is the partition represented in
    Figure \ref{fig:partition}. Hence
    \[
        w_n(A) \le \sum_{k=0}^{n-1} \mu(\triangle_k) .
    \]
    Using now the dynamical system defined in Section \ref{sec:ess} on
    the strip $\Sigma$, we have $\mu(\triangle_k) = \rho(\Sigma_k)$
    for all $k\geq 0$, so that
    \[
        w_n(A) \le \sum_{k=0}^{n-1} \mu(\triangle_k) =
        \sum_{k=0}^{n-1} \rho(\Sigma_k) = \sum_{k=0}^{n-1}
        \int_k^{k+1} \left( \int_0^1 \frac{1}{1+uv} du\right) dv =
        \int_0^n \frac{\log(1+v)}{v} dv \lesssim \log^2 n.
    \]
\end{proof}

To obtain an estimate from below, we use the matrix
representation of the local inverses defined in
Appendix~\ref{sec:matrices}.

\begin{lemma} \label{lem:integral} For a map
    $\psi = \phi_{\omega_0} \circ \phi_{\omega_1} \circ \dots \circ
    \phi_{\omega_{k-2}} \circ \phi_1$ with matrix representation
    \[
        M_\psi = \begin{pmatrix}
            r & s & t \\
            r_1 & s_1 & t_1 \\
            r_2 & s_2 & t_2
        \end{pmatrix}
    \]
    it holds that
    \[
        \frac{m(A)}{(r_1+s_1+t_1)(r_2+s_2+t_2)(r+s+t)} \le
        \mu(\psi(A)) \le \frac{27
          m(A)}{(r_1+s_1+t_1)(r_2+s_2+t_2)(r+s+t)}
    \]
\end{lemma}
\begin{proof} By definition of $\mu$, denoting
    $\psi(x,y) = (\psi_1(x,y),\psi_2(x,y))$,
    \[
        \mu(\psi(A)) = \iint_{\psi(A)} \frac{1}{xy} dx dy = \iint_A
        \frac{1}{\psi_1(x,y) \psi_2(x,y)} \left| J\psi (x,y) \right|
        dxdy.
    \]
    Moreover by Proposition \ref{prop-det}, we have
    \[
        \psi_1(x,y) = \frac{r_1 +s_1 x +t_1y}{r +s x +ty} , \quad
        \psi_2(x,y) = \frac{r_2 +s_2 x +t_2y}{r +s x +ty} , \quad
        J\psi (x,y) = \frac{1}{(r +s x +ty)^3}
    \]
    hence
    \[
        \mu(\psi(A)) = \iint_A \frac{1}{(r_1 +s_1 x +t_1y)(r_2 +s_2 x
          +t_2y)(r +s x +ty)} dxdy
    \]
    Since for $(x,y)\in A$ we can use $\frac 12\le x\le 1$ and
    $\frac 13\le y \le 1$, the proof is complete.
\end{proof}

We are then led to study the terms
\begin{equation}\label{termini}
    t_{\omega_0 \omega_1\dots\omega_{k-2} 1} \coloneqq
    \frac{1}{(r_1+s_1+t_1)(r_2+s_2+t_2)(r+s+t)}
\end{equation}
for the maps
$\psi = \phi_{\omega_0} \circ \phi_{\omega_1} \circ \dots \circ
\phi_{\omega_{k-2}} \circ \phi_1$ with $\omega\in \Omega_k$. We shall also write simply
 $t_{\psi}$ to shorten the notation. Thus we consider the sequence
\[
        \tau_n \coloneqq \sum_{k=1}^n \sum_{\omega\in \Omega_k}
        t_{\omega_0 \omega_1\dots\omega_{k-2} 1},
\]
which by Lemma \ref{lem:integral} satisfies
\begin{equation} \label{stima-uno-term}
    m(A)\sum_{\omega\in \Omega_k} t_{\omega_0
      \omega_1\dots\omega_{k-2} 1} \le \mu(A^\complement  \cap \{ \varphi =k\})
    \le 27m(A)\sum_{\omega\in \Omega_k} t_{\omega_0
      \omega_1\dots\omega_{k-2} 1}.
\end{equation}
and then
\begin{equation}\label{stima-uno}
m(A)\, \tau_n \le w_n(A) \le 27\, m(A)\, \tau_n\, .
\end{equation}
Moreover, given a linear fractional map $\psi$ with matrix representation
\[
    M_\psi = \begin{pmatrix}
        r & s & t \\
        r_1 & s_1 & t_1 \\
        r_2 & s_2 & t_2
    \end{pmatrix}
\]
if we introduce the vector
\[
    V_\psi \coloneqq
    \begin{pmatrix}
        r + s + t \\
        r_1 + s_1 + t_1 \\
        r_2 + s_2 + t_2
    \end{pmatrix},
\]
the term $t_\psi$ in \eqref{termini} is the inverse of the
product of the components of $V_\psi$. We also use the notation
$t_{V_\psi}$ for $t_\psi$.

We now define a tree $\VV$ of vectors, in such a way that the $k$-th
level of $\VV$ is associated to the set $A^\complement \cap \{\varphi=k\}$. We
first make a small modification in order to simplify the argument. For
each $k\ge 1$, we consider the subsets
\[
    \Phi_k \coloneqq A^\complement \cap \{ \varphi=k\} \cap \Gamma_1,
\]
so that
\[
    \Phi_k = \bigcup_{\omega \in \Omega_k} \phi_1 \circ
    \phi_{\omega_1} \circ \dots \circ \phi_{\omega_{k-2}} \circ \phi_1
    (A)
\]
Obviously $\Phi_1 = A^\complement \cap \{ \varphi=1\}=\{\phi_1(A)\}$, whereas for
example
\[
    \Phi_2 = \left\{ \phi_1 \circ \phi_1 (A)\right\} \subsetneq
    A^\complement \cap \{ \varphi=2\} = \left\{ \phi_0 \circ \phi_1 (A),\, \phi_1
      \circ \phi_1 (A)\right\}.
\]
We are now ready to introduce the levels of our tree $\VV$. For each
$k\geq 1$ we define
\[
    L_k \coloneqq \set{V_\psi \, :\, \psi = \phi_1 \circ
      \phi_{\omega_1} \circ \dots \circ \phi_{\omega_{k-2}} \circ
      \phi_1}\quad\text{and}\quad \lambda_k \coloneqq \sum_{V \in L_k}
    t_V,
\]
where $t_V$ is the inverse of the product of the components of the
vector $V$. The $k$-th row of $\VV$ is the set $L_k$. We have then
associated two objects to each set $A^\complement \cap \{\varphi=1\}$: the list
of vectors $L_k$ and the quantity $\lambda_k$. For instance,
corresponding to $\Phi_1$ we obtain
\[
    V_1 \coloneqq V_{\phi_1} =
    \begin{pmatrix}
        2 \\ 1 \\ 1
    \end{pmatrix}
\]
and $\lambda_1 = t_1=\frac 12$. The vector $V_1$ is the root of our
tree $\VV$. Then
\[
    L_2 = \set{ V_{\phi_1 \circ \phi_1} = \begin{pmatrix} 3 \\ 1 \\ 1
      \end{pmatrix}} \quad\text{and}\quad L_3 = \set{ V_{\phi_1 \circ
        \phi_1 \circ \phi_1} = \begin{pmatrix}
          4 \\
          1 \\
          1
      \end{pmatrix} , V_{\phi_1 \circ \phi_0 \circ \phi_1}
      = \begin{pmatrix}
          4 \\
          2 \\
          1
      \end{pmatrix}},
\]
as follows by writing the matrix representation of the involved
maps. Furthermore, for the first rows, one easily finds
$\lambda_2 = t_{V_{\phi_1 \circ \phi_1} } = \frac 13$,
$\lambda_3 = t_{V_{\phi_1 \circ \phi_1 \circ \phi_1} } + t_{V_{\phi_1
    \circ \phi_0 \circ \phi_1} } = \frac 14 + \frac 18$, and so on.

Moreover the tree $\VV$ can be generated from the root
vector $V_1$ by the following algorithm, without using the maps
$\psi$. Let us consider the matrices
\[
    M_1 = \begin{pmatrix} 1 & 0 & 1 \\ 0 & 1 & 0 \\ 0 & 0 &
        1 \end{pmatrix} \quad \text{and} \quad M_{10}= \begin{pmatrix}
        1 & 1 & 1 \\ 1 & 0 & 0 \\ 0 & 1 & 0 \end{pmatrix},
\]
that are the matrix representations of the maps $\phi_1$ and
$\phi_1\circ \phi_0$ respectively. Let them act on the vectors of the
tree to generate new vectors. When we apply $M_1$ to a vector
$V\in L_k$, we obtain a vector in $L_{k+1}$, and when we apply
$M_{10}$ we obtain a vector in $L_{k+2}$. Hence, vectors in the $k$-th
row of $\VV$ are generated by applying $M_1$ to all vectors in the
$(k-1)$-th row and $M_{10}$ to all vectors in the $(k-2)$-th
row. Applying this algorithm starting from $L_1=\set{V_1}$, we
immediately obtain for the first rows
\[
    L_2 = \set{ M_1 V_1 =  \begin{pmatrix} 1 & 0 & 1 \\ 0 & 1 & 0 \\ 0 & 0 &
        1 \end{pmatrix}\begin{pmatrix}
          2 \\
          1 \\
          1
      \end{pmatrix} = \begin{pmatrix}
          3 \\
          1 \\
          1
      \end{pmatrix}}
\]
\[
    L_3 = \set{ M_1(M_1 V_1) =  \begin{pmatrix} 1 & 0 & 1 \\ 0 & 1 & 0 \\ 0 & 0 &
        1 \end{pmatrix}\begin{pmatrix}
          3 \\
          1 \\
          1
      \end{pmatrix} = \begin{pmatrix}
          4 \\
          1 \\
          1
      \end{pmatrix},\ M_{10}V_1 = \begin{pmatrix}
        1 & 1 & 1 \\ 1 & 0 & 0 \\ 0 & 1 & 0 \end{pmatrix}\begin{pmatrix}
          2 \\
          1 \\
          1
      \end{pmatrix} = \begin{pmatrix}
          4 \\
          2 \\
          1
      \end{pmatrix}},
\]
as above.

\begin{lemma}\label{lemma:nuova-tau}
    For $n\geq 1$ define
    \[
        \tilde \tau_n \coloneqq \sum_{k=1}^n \lambda_k = \sum_{k=1}^n
        \sum_{V \in L_k} t_V.
    \]
    Then $\tilde \tau_n < \tau_n < \tilde \tau_n +
        \frac{\mu(\Gamma_0)}{m(A)}$ and $\tilde \tau_n \gtrsim \log^2 n$.
\end{lemma}
\begin{proof}
    The difference between $\tilde \tau_n$ and $\tau_n$ is that for
    each $k=1,\,\ldots,\,n$, in $\tilde \tau_n$ we are not considering
    the terms $t_{V_\psi}$ for the maps
    $\psi = \phi_{\omega_0} \circ \phi_{\omega_1} \circ \dots \circ
    \phi_{\omega_{k-2}} \circ \phi_1$ with $\omega_0=0$. Recalling
    that for such maps $t_{V_\psi} \le \frac{\mu(\psi(A))}{m(A)}$ by
    Lemma~\ref{lem:integral}, that $\psi(A) \subseteq \Gamma_0$ if
    $\omega_0=0$, and that the sets $\psi(A)$ are disjoint for
    different maps $\psi$ by definition, for all $n\ge 1$ we have that
\[
        \tilde \tau_n < \tau_n < \tilde \tau_n +
        \frac{\mu(\Gamma_0)}{m(A)}.
\]
    We prove by induction that each row $L_k$ with $k\ge 2$ contains
    the vectors
    \[
        \begin{pmatrix}
            k+1 \\
            j \\
            1
        \end{pmatrix}\quad j=1,\,\dots,\,k-1.
    \]
    By the definition of $\lambda_k$, this implies that  $\tilde \tau_n \coloneqq \sum_{k=1}^n \lambda_k \ge \sum_{k=1}^n\, \frac{1}{k+1}\, \sum_{j=1}^{k-1}\, \frac 1j \gtrsim \log^2 n$. For
    $k=2$, the row $L_2$ contains only the vector $M_1 V_1$, and the
    base case is proved. Let us assume that the statement is true for
    $r=2,\,\dots,\,k$, then using the algorithm to construct $\VV$, we
    have that $L_{k+1}$ contains the vectors
    \[
        M_1 \begin{pmatrix}
            k+1 \\
            j \\
            1
        \end{pmatrix} = \begin{pmatrix}
            k+2 \\
            j \\
            1
        \end{pmatrix}\quad j=1,\,\dots,\,k-1
        \quad\text{and}\quad
        M_{10} \begin{pmatrix}
            k \\
            1 \\
            1
        \end{pmatrix} = \begin{pmatrix}
            k+2 \\
            k \\
            1
        \end{pmatrix} .
    \]
    Hence the statement is true for $L_{k+1}$.
\end{proof}

\begin{proposition}\label{prop:w-below}
    The wandering rate $w_n(A)$ satisfies
    $w_n(A) \gtrsim \log^2 n$.
\end{proposition}

\begin{proof}
It follows from \eqref{stima-uno} and Lemma \ref{lemma:nuova-tau}.
\end{proof}

Finally we discuss the property of regular variation for $w_n(A)$. The first remark is that if $w_n(A)$ is regularly varying then it is slowly varying. By \eqref{def:reg-var}, if $w_n(A)$ is regularly varying then there exists $\alpha\in \R$ such that
\[
\lim_{n\to \infty}\, \frac{w_{cn}}{w_n} = c^\alpha
\]
for all $c\in \N$. However by Propositions \ref{prop:w-above} and \ref{prop:w-below}, there exist two constants $k_1,k_2$ with $0<k_1 < 1 < k_2$ such that
\[
k_1\, \frac{\log^2 (cn)}{\log^2 (n)} \le \frac{w_{cn}}{w_n} \le k_2\, \frac{\log^2 (cn)}{\log^2 (n)}
\]
and passing to the limit we obtain
\[
k_1 \le c^\alpha \le k_2
\]
for all $c\in \N$. Hence $\alpha=0$.

A second remark is that we have a sufficient condition on the sequence $\lambda_k$ from Lemma \ref{lemma:nuova-tau} for $w_n(A)$ being slowly varying. Since $\tilde \tau_n(A) \gtrsim \log^2 n$, it is immediate that
\[
\liminf_{k\to \infty} \, \frac{k\, \lambda_k}{\log^2 k} = 0
\]
To have that $w_n(A)$ is slowly varying it is enough that also the limsup vanishes.

\begin{lemma} \label{lem:cond-suff}
If $\lambda_k = o(\frac{\log^2 k}{k})$ then $w_n(A)$ is slowly varying.
\end{lemma}

\begin{proof}
From \eqref{stima-uno-term} we obtain that if $\tau_n$ is slowly varying the same holds for $w_n(A)$. Indeed
\[
    w_{2n}(A) - w_n(A) = \sum_{k=n+1}^{2n} \mu(A^\complement  \cap \{ \varphi
    =k\}) \le 27m(A) \sum_{k=n+1}^{2n} \sum_{\omega\in \Omega_k}
    t_{\omega_0 \omega_1\dots\omega_{k-2} 1} = 27m(A) (\tau_{2n} -
    \tau_n)
\]
and
\[
    w_n(A) = \sum_{k=1}^n \mu(A^\complement  \cap \{ \varphi =k\}) \ge m(A)
    \sum_{k=1}^n \sum_{\omega\in \Omega_k} t_{\omega_0
      \omega_1\dots\omega_{k-2} 1} = m(A) \tau_n.
\]
In conclusion
\[
    1\le \frac{w_{2n}(A)}{w_n(A)} = 1+ \frac{w_{2n}(A) -
      w_n(A)}{w_n(A)} \le 1 + 27 \frac{\tau_{2n}-\tau_n}{\tau_n} .
\]
If $(\tau_n)_{n\geq 1}$ is slowly varying the term
$\frac{\tau_{2n}-\tau_n}{\tau_n}$ is vanishing, and the result
follows.

Moreover from Lemma \ref{lemma:nuova-tau} it is immediate that if $\tilde \tau_n$ is slowly varying then the same is true for $\tau_n$. We are thus reduced to study $\tilde \tau_n$. We first claim that it is enough to show that \eqref{def:reg-var} holds with $\alpha=0$ only for $c=2$ (see for example \cite[Proposition 1.10.1]{regular}). Indeed, for $1<c<2$ we write
\[
    1 \le \frac{\tilde \tau_{\lfloor cn \rfloor}}{\tilde \tau_n} \le
    \frac{\tilde \tau_{2n}}{\tilde \tau_n}.
\]
For $c>2$, let $k\ge 1$ such that $c\le 2^k$, then we write
\[
    1\le \frac{\tilde \tau_{\lfloor cn \rfloor}}{\tilde \tau_n} \le \frac{\tilde \tau_{2^k
        n}}{\tilde \tau_{2^{k-1}n}} \frac{\tilde \tau_{2^{k-1}
        n}}{\tilde \tau_{2^{k-2}n}} \cdots \frac{\tilde \tau_{2 n}}{\tilde \tau_{n}}
\]
and for all $j=1,\,\dots,\,k $ we have
$\frac{\tilde \tau_{2^j n}}{\tilde \tau_{2^{j-1}n}} \to 1$, because it is a
subsequence of $\frac{\tilde \tau_{2n}}{\tilde \tau_{n}}$. Hence \eqref{def:reg-var}
follows again with $\alpha=0$ for $c>2$. We can proceed analogously for the case
$0<c<1$, which completes the proof of the claim.

Moreover we follow the proof of \cite[Theorem 1.5.4]{regular} to show that \eqref{def:reg-var} holds with $\alpha=0$ for $c=2$. Let $\alpha>0$, then by definition the sequence $\phi(n)\coloneqq n^\alpha\, \tilde \tau_n$ is non-decreasing. We also show that the sequence $\psi(n)\coloneqq n^{-\alpha}\, \tilde \tau_n$ is eventually non-increasing. Indeed
\[
\psi(n)-\psi(n+1) = \psi(n)\, \left( 1- \frac{\tilde \tau_{n+1}}{\tilde \tau_n}\, \frac{n^\alpha}{(n+1)^\alpha} \right) = \psi(n)\, \left( 1- \frac{1+\frac{\lambda_{n+1}}{\tilde \tau_n}}{(1+\frac 1n)^\alpha} \right)
\]
and $\lambda_{n+1} = o(\frac{\log^2 (n+1)}{n+1})$ together with $\tilde \tau_n \gtrsim \log^2 n$ implies
\[
1+\frac{\lambda_{n+1}}{\tilde \tau_n} = o \left( \frac{1}{n} \right).
\]
Hence we have that
\[
1+\frac{\lambda_{n+1}}{\tilde \tau_n} < \left(1+\frac 1n\right)^\alpha = 1+\alpha \frac 1n + o \left( \frac{1}{n} \right)
\]
for $n$ big enough. It follows that $\psi(n)-\psi(n+1)\ge 0$ eventually. For $n$ big enough we can then write
\[
2^{-\alpha} = \frac{\tilde \tau_{2n}}{\tilde \tau_n}\, \frac{\phi(n)}{\phi(2n)} \le \frac{\tilde \tau_{2n}}{\tilde \tau_n} \le \frac{\tilde \tau_{2n}}{\tilde \tau_n}\, \frac{\psi(n)}{\psi(2n)} = 2^\alpha
\]
hence
\[
2^{-\alpha} \le \liminf_{n\to \infty}\, \frac{\tilde \tau_{2n}}{\tilde \tau_n}\le \limsup_{n\to \infty}\, \frac{\tilde \tau_{2n}}{\tilde \tau_n}\le 2^\alpha.
\]
Since the previous argument can be repeated for all $\alpha>0$ it follows that
\[
\lim_{n\to \infty}\, \frac{\tilde \tau_{2n}}{\tilde \tau_n} = 1\, .
\]
\end{proof}

Finally, we recall from \cite{aaronson:iet} and \cite{zwei} that for the pointwise dual ergodic system $(\trianglecl, \mu,S)$ the renormalising sequence $a_n(S)$ is defined in terms of the wandering rate of a subset on which the induced map is $\psi$-mixing (see Proposition \ref{prop:cfm-pde}), and $a_n(S)$ is asymptotically independent on the chosen subset with this property. In particular this implies that if $\Gamma_0$ satisfies Proposition \ref{prop:cfm-pde} then we have
\[
a_n(S) \asymp \frac{n}{w_n(\Gamma_0)}
\]
where
\[
w_n(\Gamma_0) = \sum_{k=0}^{n-1}\, \mu(\triangle_k) = \int_0^n\, \frac{\log(1+v)}{v}\, dv
\]
as shown in Proposition \ref{prop:w-above}. Since it is not difficult to show that $w_n(\Gamma_0)$ is slowly varying, then we would have
\[
a_n(S) \sim \frac{n}{\int_0^n\, \frac{\log(1+v)}{v}\, dv}
\]
as discussed in Remark \ref{rem:slowly}. Unfortunately it is not known whether $\Gamma_0$ is a \emph{good set} to which apply Proposition \ref{prop:cfm-pde}.


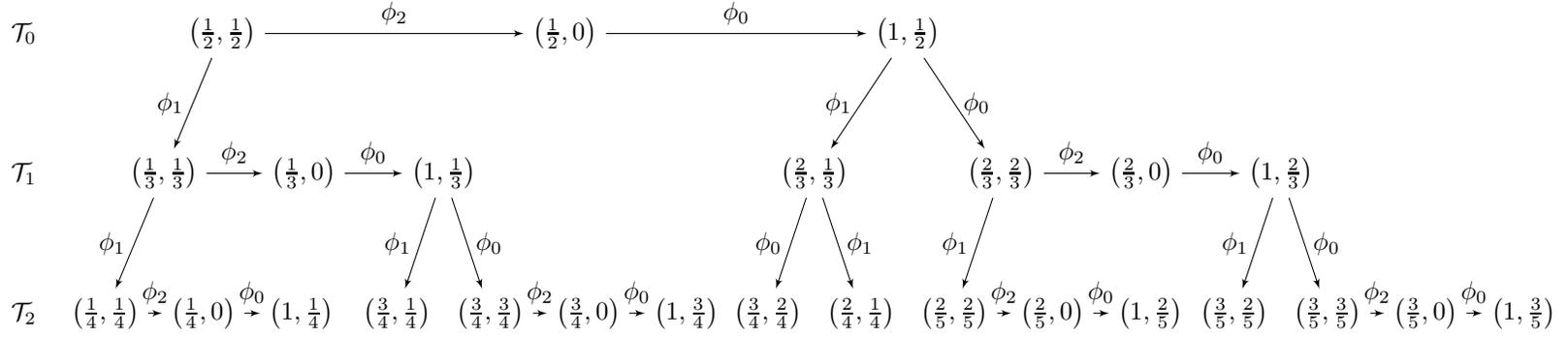
\begin{sidewaysfigure}
    \vspace{15cm}
    \begin{tikzpicture}[
    level 1/.style = {sibling distance=2.5cm},
    level 2/.style = {sibling distance=2.5cm},
    level 3/.style = {sibling distance=4cm},
    level 4/.style = {sibling distance=2cm},
    level distance = 3cm,
    edge from parent/.style = {draw,-latex'},
    every node/.style       = {-latex'},
    scale=0.64
    ]
    \def\d{4.35}
    \def\hd{0} 
    \def\hdd{-0.9}   
    \def\ld{2.5}   

    \node at (-4.25,0) {$\mathcal{T}_0$};
    \node at (-4.25,-3) {$\mathcal{T}_1$};
    \node at (-4.25,-6) {$\mathcal{T}_2$};

    \node {$\triple{1}{1}{2}$} [grow=down]
    child {
      node {$\triple{1}{1}{3}$}
      child {
        node {$\triple{1}{1}{4}$}
        child [grow=right] {
          node at (\hdd,0) {$\singy{1}{4}{0}$}
          child [grow=right] {
            node at (\hdd,0) {$\singx{1}{1}{4}$} [grow=down]
            edge from parent node[above] {$\phi_0$}}
          edge from parent node[above] {$\phi_2$}}
        edge from parent node[left] {$\phi_1$}}
      child [grow=right] {
        node at (\hd,0) {$\singy{1}{3}{0}$}
        child [grow=right] {
          node  at (\hd,0) {$\singx{1}{1}{3}$} [grow=down]
          child {
            node {$\triple{3}{1}{4}$}
            edge from parent node[left] {$\phi_1$}}
          child {
            node {$\triple{3}{3}{4}$}
            child [grow=right] {
              node at (\hdd,0) {$\singy{3}{4}{0}$}
              child [grow=right] {
                node at (\hdd,0) {$\singx{1}{3}{4}$} [grow=down]
                edge from parent node[above] {$\phi_0$}}
              edge from parent node[above] {$\phi_2$}}
            edge from parent node[right] {$\phi_0$}}
          edge from parent node[above] {$\phi_0$}
        }
        edge from parent node[above] {$\phi_2$}
      }
      edge from parent node[anchor=east] {$\phi_1$}
    }
    child [grow=right] {
      node at (\d,0) {$\singy{1}{2}{0}$}
      child [grow=right] {
        node  at (\d,0) {$\singx{1}{1}{2}$} [grow=down]
        child {
          node {$\triple{2}{1}{3}$} [grow=down]
          child {
            node {$\triple{3}{2}{4}$}
            edge from parent node[left] {$\phi_0$}}
          child {
            node {$\triple{2}{1}{4}$}
            edge from parent node[right] {$\phi_1$}}
          edge from parent node[left] {$\phi_1$}}
        child {
          node {$\triple{2}{2}{3}$}
          child  {
            node {$\triple{2}{2}{5}$}
            child [grow=right] {
              node at (\hdd,0) {$\singy{2}{5}{0}$}
              child [grow=right] {
                node at (\hdd,0) {$\singx{1}{2}{5}$} [grow=down]
                edge from parent node[above] {$\phi_0$}}
              edge from parent node[above] {$\phi_2$}}
            edge from parent node[left] {$\phi_1$}}
          child [grow=right] {
            node at (\hd,0) {$\singy{2}{3}{0}$}
            child [grow=right] {
              node at (\hd,0) {$\singx{1}{2}{3}$} [grow=down]
              child {
                node {$\triple{3}{2}{5}$}
                edge from parent node[left] {$\phi_1$}}
              child {
                node {$\triple{3}{3}{5}$}
                child [grow=right] {
                  node at (\hdd,0) {$\singy{3}{5}{0}$}
                  child [grow=right] {
                    node at (\hdd,0) {$\singx{1}{3}{5}$}
                    edge from parent node[above] {$\phi_0$}}
                  edge from parent node[above] {$\phi_2$}}
                edge from parent node[right] {$\phi_0$}}
              edge from parent node[above] {$\phi_0$}
            }
            edge from parent node[above] {$\phi_2$}
          }
          edge from parent node[right] {$\phi_0$}}
        edge from parent node[above] {$\phi_0$}
      }
      edge from parent node[above] {$\phi_2$}
    };
\end{tikzpicture}
    \caption{The first three levels of the tree generated through the
      local inverses of the map $\tilde{S}$.}\label{fig:tree}
\end{sidewaysfigure}

\end{document}